\documentclass[10pt, a4paper]{amsart}

\usepackage{amssymb,amsfonts}
\usepackage[all,arc]{xy}
\usepackage{enumerate}
\usepackage{mathrsfs}
\usepackage{tikz-cd}
\usepackage{enumitem}
\usepackage{empheq}
\usepackage{kotex}
\usepackage{ifpdf}
\usepackage{mathtools}
\usepackage{pst-node, pst-plot, pstricks}
\usepackage{mathrsfs}
\usepackage{amsmath}
\usepackage{epsfig}
\usepackage{amscd}
\usepackage{graphicx, color}
\graphicspath{ {./images/} }

\def\E{\ifmmode{\mathbb E}\else{$\mathbb E$}\fi} %natural numbers
\def\N{\ifmmode{\mathbb N}\else{$\mathbb N$}\fi} %natural numbers%
\def\R{\ifmmode{\mathbb R}\else{$\mathbb R$}\fi} %real numbers
\def\Q{\ifmmode{\mathbb Q}\else{$\mathbb Q$}\fi} %rational numbers
\def\C{\ifmmode{\mathbb C}\else{$\mathbb C$}\fi} %complex numbers
\def\H{\ifmmode{\mathbb H}\else{$\mathbb H$}\fi} %complex numbers
\def\Z{\ifmmode{\mathbb Z}\else{$\mathbb Z$}\fi} %integers
\def\P{\ifmmode{\mathbb P}\else{$\mathbb P$}\fi} %real numbers
\def\T{\ifmmode{\mathbb T}\else{$\mathbb T$}\fi} %real numbers
\def\SS{\ifmmode{\mathbb S}\else{$\mathbb S$}\fi} %real numbers
\def\DD{\ifmmode{\mathbb D}\else{$\mathbb D$}\fi} %real numbers

\DeclareSymbolFont{yhlargesymbols}{OMX}{yhex}{m}{n}

\DeclareMathAccent{\widetriangle}{\mathord}{yhlargesymbols}{"E6}

\def\E{\ifmmode{\mathbb E}\else{$\mathbb E$}\fi} %natural numbers
\def\N{\ifmmode{\mathbb N}\else{$\mathbb N$}\fi} %natural numbers%
\def\R{\ifmmode{\mathbb R}\else{$\mathbb R$}\fi} %real numbers
\def\Q{\ifmmode{\mathbb Q}\else{$\mathbb Q$}\fi} %rational numbers
\def\C{\ifmmode{\mathbb C}\else{$\mathbb C$}\fi} %complex numbers
\def\H{\ifmmode{\mathbb H}\else{$\mathbb H$}\fi} %complex numbers
\def\Z{\ifmmode{\mathbb Z}\else{$\mathbb Z$}\fi} %integers
\def\P{\ifmmode{\mathbb P}\else{$\mathbb P$}\fi} %real numbers
\def\T{\ifmmode{\mathbb T}\else{$\mathbb T$}\fi} %real numbers
\def\SS{\ifmmode{\mathbb S}\else{$\mathbb S$}\fi} %real numbers
\def\DD{\ifmmode{\mathbb D}\else{$\mathbb D$}\fi} %real numbers

\newcommand{\ben}{\begin{enumerate}}
\newcommand{\een}{\end{enumerate}}
\newcommand{\be}{\begin{equation}}
\newcommand{\ee}{\end{equation}}
\newcommand{\bea}{\begin{eqnarray}}
\newcommand{\eea}{\end{eqnarray}}
\newcommand{\beastar}{\begin{eqnarray*}}
\newcommand{\eeastar}{\end{eqnarray*}}
\newcommand{\bc}{\begin{center}}
\newcommand{\ec}{\end{center}}

\theoremstyle{theorem}
\newtheorem{thm}{Theorem}[section]
\newtheorem{cor}[thm]{Corollary}
\newtheorem{lem}[thm]{Lemma}
\newtheorem{prop}[thm]{Proposition}
\newtheorem{'thm'}[thm]{'Theorem'}

\theoremstyle{definition}
\newtheorem{defn}[thm]{Definition}
\newtheorem{rem}[thm]{Remark}

\newtheorem{exam}[thm]{Example}

\newtheorem{proof-sketch}[thm]{Proof-Sketch}

\newtheorem{lem-defn}[thm]{Lemma-Definition}
\newtheorem{prop-defn}[thm]{Proposition-Definition}
\newtheorem{thm-defn}[thm]{Theorem-Definition}
\newtheorem{assump}[thm]{Assumption}

\newtheorem*{thm*}{Theorem}

\numberwithin{equation}{section}

\def\R{{\mathbb R}}

\def\E{{\mathbb E}}
\def\Z{{\mathbb Z}}
\def\C{{\mathbb C}}
\def\R{{\mathbb R}}
\def\P{{\mathbb P}}

\def\N{{\mathbb N}}

\def\11{{\mathbb I}}

\def\sgn{{\text{\rm sgn}}}

\def\C{\mathbb{C}}
\def\Z{\mathbb{Z}}

\def\T{\mathbb{T}}

\def\Q{\mathbb{Q}}

\def\E{\ifmmode{\mathbb E}\else{$\mathbb E$}\fi} %natural numbers
\def\N{\ifmmode{\mathbb N}\else{$\mathbb N$}\fi} %natural numbers
\def\R{\ifmmode{\mathbb R}\else{$\mathbb R$}\fi} %real numbers
\def\Q{\ifmmode{\mathbb Q}\else{$\mathbb Q$}\fi} %rational numbers
\def\C{\ifmmode{\mathbb C}\else{$\mathbb C$}\fi} %complex numbers
\def\H{\ifmmode{\mathbb H}\else{$\mathbb H$}\fi} %complex numbers
\def\Z{\ifmmode{\mathbb Z}\else{$\mathbb Z$}\fi} %integers
\def\P{\ifmmode{\mathbb P}\else{$\mathbb P$}\fi} %real numbers
\def\SS{\ifmmode{\mathbb S}\else{$\mathbb S$}\fi} %real numbers
\def\DD{\ifmmode{\mathbb D}\else{$\mathbb D$}\fi} %real numbers

\def\R{{\mathbb R}}

\def\E{{\mathbb E}}
\def\Z{{\mathbb Z}}
\def\C{{\mathbb C}}
\def\R{{\mathbb R}}

\def\N{{\mathbb N}}
\def\CK{{\mathcal K}}

\def\darr#1{\raise1.5ex\hbox{$\leftrightarrow$}
\mkern-16.5mu #1}

\def\roughly#1{\raise.3ex\hbox{$#1$\kern-.75em
\lower1ex\hbox{$\sim$}}}

\def\opname#1{\mathop{\kern0pt{\rm #1}}\nolimits}

\def\dim{\opname{dim}}

\def\Incl{\operatorname{Incl}}
\def\Eval{\operatorname{Eval}}

\begin{document}

\quad \vskip1.375truein

\bibliographystyle{plain}

%--------Meta Data: Fill in your info------
\title[{Kuranishi chart categories and higher cocycle conditions}]
{Kuranishi chart categories and higher cocycle conditions}

\author{Taesu Kim}
\address{}
\thanks{}

%\date{March, 2021; Revised in April, 2021}
\begin{abstract}
Given an $L_\infty$-Kuranishi space introduced in \cite{Kim1}, we propose a notion called the Kuranishi chart category. Using the nerve of this category, together with a choice of atlas and a simplicial description of the covering of the underlying topological space, we formulate a higher homotopical version of the bundle-component cocycle condition. We show that this condition is always satisfied, by virtue of a property of the higher homotopy theory of $L_\infty[1]$-morphisms developed in \cite{Kim2}, concerning quasi-isomorphisms. As a consequence, the rigid cocycle condition of Fukaya-Oh-Ohta-Ono Kuranishi spaces is replaced by more flexible, homotopy-theoretic compatibility.
\end{abstract}

\keywords{$L_{\infty}[1]$-algebras, $L_{\infty}$-Kuranishi spaces, Kuranishi chart categories, Simplicial sets, Hypercoverings, Higher homotopies, Higher cocycle conditions}
\subjclass[2020]{Primary 18N50; Secondary 55U35, 53D35}

\maketitle

\tableofcontents

\section{Introduction}
In \cite{Kim1}, a new definition of Kuranishi spaces associated with $L_{\infty}[1]$-structures is introduced, providing the existing theory with a flexible point of view based on $L_{\infty}[1]$-homotopy theory. There, the cocycle condition for the coordinate changes is imposed only on the base maps, and not on the bundle component, or the $L_{\infty}[1]$-component  (cf. \cite[Remark 3.2 (iii)]{Kim1}). In relation to this point, we can demonstrate that the coordinate changes in \cite{Kim1} admit yet another layer of flexibility in terms of higher homotopy theory, and we declare that $L_{\infty}$-compatibilities always hold. The purpose of this paper is to make these arguments precise and rigorous.

We first briefly recall the definitions of the Fukaya-Oh-Ohta-Ono charts and coordinate changes following \cite{FOOO1}. Let $X$ be a compact metrizable space. To each point $p \in X$, we assign a chart $\mathscr{U}_p := (U_p, E_p, s_p, \Gamma_p, \psi_p)$, where $\pi: E_p \rightarrow U_p$ is a vector bundle over a smooth manifold $U_p$ and $s_p$ is its section, and suppose that a finite group $\Gamma_p$ acts on this data. Then there exists a neighborhood of each point in $X$ that is homeomorphic to the quotient $s_p^{-1}(0)/\Gamma_p$.

To each pair of points $p , q \in X,$ with $q \in X$ and $p \in \mathrm{Im}\psi_q,$ an FOOO \textit{coordinate change} from $\mathscr{U}_p$  to  $ \mathscr{U}_q$ is given by a tuple $\Phi_{pq}=(U_{pq} , \phi_{pq}, \widetilde{\phi}_{pq}).$ Here $U_{pq} \subset U_p$ is an open subset, while $\phi_{pq}: U_{pq} \hookrightarrow U_q$ is an embedding and $\widetilde{\phi}_{pq} :  E_p|_{U_{pq}} \hookrightarrow E_q$ is a bundle embedding. They are required to satisfy natural compatibility conditions with respect to the sections $s_p, s_{q}$ and the Kuranishi maps $\psi_p, \psi_q,$ together with the tangent bundle condition (cf. \cite[Definition 3.2]{FOOO1} and \cite[Definition 2.27]{Kim1}).

This definition, however, has been a source of rigidity, preventing the theory from admitting a categorical structure. In \cite{Kim1}, to resolve this, the bundle component is replaced by an $L_{\infty}[1]$-quasi-isomorphism. In particular, the tangent bundle condition yields the quasi-isomorphism property. Moreover, the Whitehead theorem for $L_{\infty}[1]$-algebras proves useful in dealing with several technical issues therein.

In this paper, we further claim that the cocycle condition for the FOOO coordinate changes can be replaced by a relaxed version accordingly. Recall that to obtain the Fukaya-Oh-Ohta-Ono Kuranishi space, we impose the cocycle condition on a collection of charts and their coordinate changes:
\[
\Phi_{pr}|_{U_{pqr}} = \Phi_{qr} \circ \Phi_{pq}|_{U_{pqr}}
\]
for $q \in \mathrm{Im}\, \psi_p$, $r \in \psi_q \big(s_q^{-1}(0) \cap U_{qr}\big)$, and the commonly defined region $U_{pqr}$. In particular, the bundle component corresponds to the equation between the two bundle morphisms
\begin{equation}\label{agjalksfdlks}
\widetilde{\phi}_{pr}|_{U_{pqr}} = \widetilde{\phi}_{qr} \circ \widetilde{\phi}_{pq}|_{U_{pqr}}.
\end{equation}
We can relax this rigid notion of an on-the-nose cocycle condition to a higher-homotopical counterpart. Roughly speaking, (\ref{agjalksfdlks}) is rewritten as an equality between $L_{\infty}$-morphisms up to filling homotopies.

To make the above statement rigorous, we employ the higher homotopy theory of $L_{\infty}[1]$-morphisms developed in \cite{Kim2}. More concretely, when patching the $L_{\infty}$-information at each chart for an $L_{\infty}$-Kuranishi atlas/space, we need a framework for systematically handling multiple intersections of Kuranishi charts together with coordinate changes, their homotopies, homotopies of homotopies, and so on. We denote by $N\left(\widehat{\mathcal{U}}\right)_{\bullet}$ the simplicial set that naturally captures the intersection data among the bases of the charts for a given atlas $\widehat{\mathcal{U}}$, by assigning a simplex $\alpha$ indexing the intersected open subset $U_{\alpha}$. A \textit{Kuranishi hypercovering} is then defined as the hypercovering that this simplicial set induces on the topological space covered by the atlas $\widehat{\mathcal{U}}$.

We then associate an $L_{\infty}$-Kuranishi space $\mathfrak{X}=\left(X, \left[\widehat{\mathcal{U}}\right]\right)$ with a simplicially enriched category $\mathcal{K}_{\mathfrak{X}}$ called \textit{the chart category} of ${\mathfrak{X}}$ whose objects are given by all Kuranishi charts that belong to an atlas of the same equivalence class as $\left[\widehat{\mathcal{U}}\right].$ For a pair $\mathcal{U}_p$ and $\underline{\mathcal{U}}_q \in \text{Ob}(\mathcal{K}_{\mathfrak{X}}),$ with $\mathrm{Im}\psi_p \cap \mathrm{Im}\psi_q \neq \emptyset,$ their morphism space is given by
\begin{equation}\nonumber
\text{Mor}_{\mathcal{K}_{\mathfrak{X}}}(\mathcal{U}_p, \underline{\mathcal{U}}_q):= \coprod\limits_{k=0}^{\infty} M_{pq}^k,
\end{equation}
where $M_{pq}^k$'s are roughly speaking a compatible family of $k$-standard simplex parameterized coordinate changes. (See Subsection 4.3. for the precise description.)

There exists a filtration-like structure $\left\{N_{\bullet}\left(\mathcal{K}_{\mathfrak{X}}^{(l)}\right)\right\}_l$ of $N_{\bullet}\left(\mathcal{K}_{\mathfrak{X}}\right)$ according to the base dimension of each chart, together with the naturally defined embedding of simplicially enriched categories 
\[
\mathscr{I}^{(l)} : N_{\bullet}\left(\mathcal{K}_{\mathfrak{X}}^{(l)}\right), \hookrightarrow N_{\bullet}\left(\mathcal{K}_{\mathfrak{X}}^{(l+1)}\right)
\] 
for each $l.$ The higher compatibility is then written as a map from the simplicial set  $N\left(\widehat{\mathcal{U}}\right)_{\bullet}$ to the simplicial nerve construction of $\mathcal{K}_{\mathfrak{X}}$. For $l \geq 1$, we consider a simplicial map
\begin{equation}\nonumber
\mathscr{G}_{\bullet}^{(l)} : N\left(\widehat{\mathcal{U}}\right)_{\bullet} \rightarrow N_{\bullet}\left(\mathcal{K}_{\mathfrak{X}}^{(l)}\right),
\end{equation}
and call it a \textit{higher cocycle condition} of the Kuranishi space $\mathfrak{X}$ if the family satisfies the following compatibilities:
\begin{enumerate}[label=(\roman*)]
\item $\mathscr{G}^{(l)}_{m-1}(\partial_j \alpha)\big|_{U_{\alpha}} = \partial_j \mathscr{G}^{(l)}_m(\alpha)$, $j = 0, \cdots, m$,
\item $\mathscr{G}^{(l)}_{m+1}(\sigma_j \alpha) = \sigma_j\mathscr{G}^{(l)}_m(\alpha)$, $j = 0, \cdots, m$,
\item $\mathscr{G}^{(l+1)}_{\bullet} = \mathscr{I}^{(l)} \circ \mathscr{G}^{(l)}_{\bullet}$.
\end{enumerate}

Indeed, the homotopy theory of $L_{\infty}[1]$-morphisms becomes virtually trivial with respect to quasi-isomorphisms, in the sense that any simplex of intersecting charts whose vertices are assigned quasi-isomorphisms can be filled by a higher homotopy (cf.\ Corollary \ref{anhp}). By virtue of this fact, the family $\left\{\mathscr{G}_{\bullet}^{(l)}\right\}_{l \geq 0}$ can be constructed recursively, and we have:

\begin{thm}[Higher cocycle conditions]
Given a Kuranishi space and a choice of its atlas, higher cocycle conditions always exist.
\end{thm}

The cost of this approach is that we must explicitly assign a filling homotopy to each simplex to complete the picture. Thus, the $L_{\infty}[1]$-compatibilities are governed by data rather than conditions.

We outline the structure of this paper. Section 2 reviews the content of \cite{Kim1}. Namely, the definition of $L_{\infty}$-Kuranishi spaces and their category is introduced. Section 3 is devoted to the construction of the hypercoverings and the Kuranishi chart category associated to an $L_{\infty}$-Kuranishi space/atlas. In Section 4, we rigorously formulate the cocycle conditions for the $L_{\infty}[1]$-component of coordinate changes. In Appendix A, we briefly recall the higher homotopy theory of $L_{\infty}[1]$-morphisms as developed in \cite{Kim2}. In Appendix B, we revisit a more detailed description of the structures placed on neighborhoods of the zero locus following \cite{Kim1}.

\section{$L_{\infty}$-Kuranishi spaces and their category}

In this section, we recall the definition of $L_{\infty}$-Kuranishi spaces in \cite{Kim1} and show that their collection form a category where the category of smooth manifolds embeds.

\subsection{$L_{\infty}$-Kuranishi charts}

We define Kuranishi charts by associating $L_{\infty}[1]$-algebras to presymplectic neighborhoods of the zero locus of the Kuranishi section.

\begin{defn}[$L_\infty$-Kuranishi charts]\label{kurdef}
Let $X$ be a compact metrizable space. An $L_\infty$-\textit{Kuranishi chart} of $X$ is given by a tuple
\begin{equation}\nonumber
\mathcal{U} = (U, E, s, \Gamma, \psi),
\end{equation}
where
\begin{enumerate}
\item[--] $U = (U, \beta)$ is a pair of a smooth manifold with a closed two-form $\beta \in Z^2(U).$
\item[--] $\pi : E \rightarrow U$ is a (finite rank) vector bundle.
\item[--] $s : U \rightarrow E$ is a smooth section.
\item[--] $\Gamma$ is a finite group acting on $U$ that restricts to the zero set of $s$, that is, $\Gamma \cdot s^{-1}(0) \subset s^{-1}(0).$
\item[--] $\psi : {s^{-1}(0)}/{\Gamma} \overset{\simeq}{\hookrightarrow} X$ is a homeomorphism onto the image.
\end{enumerate}
The \textit{dimension} of $\mathcal{U}$ is defined by $\dim \mathcal{U} := \dim U - \text{rk}E.$

We require that the chart  $\mathcal{U}$ be endowed with the following structures:
\begin{itemize}
\item[--] $U$ has a decomposition 
\begin{equation}\label{wstri}
U = \bigcup\limits_i \mathcal{S}_i,
\end{equation}
into (possibly non-connected) submanifolds,
\[
\mathcal{S}_i := \{x \in U \mid \text{rk} (\ker\beta_x) =i \}, \ 0 \leq i \leq \dim U,
\]
together with their tubular neighborhoods for each $i$.

\item[--] To each zero point $x \in s^{-1}(0),$ we assign:
\begin{enumerate}[label = (\roman*)]
\item A presymplectic open neighborhood $W_x$ of $x$ in $U$ with $W_x \simeq B^n,$
\item A local $ L_{\infty}[1]$-algebra $\mathcal{C}_x$
\[
\mathcal{C}_x := \overbrace{\bigwedge\nolimits^{-\bullet}\Gamma(E^*|_{W_x})}^{\text{Koszul}} \oplus \overbrace{\Omega^{\bullet + 1}_{\mathrm{aug}}(\mathcal{F}_x)}^{\text{de Rham}},
\]
\end{enumerate}
whose detailed descriptions are provided in Appendix B.
\end{itemize}
\end{defn}

\subsection{Morphisms of charts and Kuranishi atlases}

We define morphisms of Kuranishi charts, with embeddings as a special case, which yields the construction of a global object called a Kuranishi atlas.

Let $f : X \rightarrow X'$ be a continuous map between compact topological spaces. 

\begin{defn}[Chart morphisms]
A \textit{chart morphism} between $L_{\infty}$-Kuranishi charts ${\mathcal{U}} = (U,E,\Gamma,s,\psi)$ and ${\mathcal{U}}' = (U',E',\Gamma',s',\psi')$ of $X$ and $X',$ respectively, is defined by a pair
\[
\Phi = \left(\phi, \widehat{\phi}\right) : {\mathcal{U}} \rightarrow {\mathcal{U}}' ,
\]
given by:
\begin{enumerate}
\item[--] $\phi : U \rightarrow U', \text{ a } (\Gamma, \Gamma')\text{-equivariant map of manifolds,}$
\item[--] $\widehat{\phi} = \left\{\widehat{\phi}_x :  \mathcal{C}'_{\phi(x)} \rightarrow \mathcal{C}_x \right\}_{x \in s^{-1}(0)}$ is a family of $L_{\infty}[1]$-morphisms,
\end{enumerate}
satisfying 
\begin{enumerate}[label = (\roman*)]
\item $\psi' \circ \phi = f \circ \psi$ on $s^{-1}(0),$
\item $\phi(W_x) \subset W'_{\phi(W_x)},$
\item $\widehat{\phi}_x$ factors through $\mathcal{C}'_{\phi(x), \phi},$ that is, we have $\widehat{\phi}_x = \widehat{\phi}_x^{\mathrm{c}} \circ \widehat{\varepsilon}_{\phi(x), \phi}$ for some $L_{\infty}[1]$-morphism, $\widehat{\phi}^{\mathrm{c}}_x :  \mathcal{C}'_{\phi(x), \phi} \rightarrow \mathcal{C}_x.$
\end{enumerate}
Here, $\mathcal{C}_{\phi(x),\phi}'$ stands for \textit{the completion} of $\mathcal{C}'_{\phi(x)}$ \textit{at the image of} $\phi$, induced from the \textit{completed} V-algebra at the image (see Definition\ref{ladef} for details). We remark that for surjective $\phi$, we have an isomorphism $\mathcal{C}'_{\phi(x), \phi} \simeq \mathcal{C}'_{\phi(x)}$ (cf.\ Lemma \ref{surjcp}).

On the other hand, $\widehat{\varepsilon}_{\phi(x), \phi} : \mathcal{C}'_{\phi(x)} \rightarrow \mathcal{C}'_{\phi(x), \phi}$ is the $L_{\infty}[1]$-morphism given by considering elements in $\mathcal{C}'_{\phi(x)}$ naturally as those in $\mathcal{C}'_{\phi(x), \phi}$, which can be easily shown to define an $L_{\infty}[1]$-morphism. (See Definition \ref{varep} and Lemma \ref{vecpf}.)
\end{defn}

\begin{defn}[Embedding of charts]\label{ourcemb}
A chart morphism $\left(\phi, \widehat{\phi}\right)$ is called an \textit{embedding} if 
\begin{enumerate}
\item[--]$\phi$ is an (equivariant) embedding of manifolds, 
\item[--] $\widehat{\phi}^{\mathrm{c}}_{x} : \mathcal{C}'_{\phi(x), \phi} \rightarrow \mathcal{C}_x$ is a quasi-isomorphism of $L_{\infty}[1]$-algebras for each $x.$
\end{enumerate}
\end{defn}

A \textit{coordinate change} of $L_{\infty}$-Kuranishi charts provides a main example of embedding:

\begin{defn}[Coordinate changes]\label{kstr}
For two points $p, q \in X$ with $\mathrm{Im}\psi_p \cap \mathrm{Im}\psi_q \neq \emptyset,$ we define a \textit{coordinate change} $\Phi_{pq} : \mathcal{U}_p \rightarrow \mathcal{U}_q$ by a tuple
\[
\Phi_{pq} := \left(U_{pq}, \phi_{pq}, \left\{\widehat{\phi}_{pq,x}\right\}\right),
\]
where $U_{pq} \subset U_p$ is an open submanifold, and
\[
\left(\phi_{pq}, \widehat{\phi}_{pq}\right) : \mathcal{U}_p|_{U_{pq}} \rightarrow \mathcal{U}_q
\]
is an embedding of $L_{\infty}$-Kuranishi charts from $\mathcal{U}_p|_{U_{pq}},$ that is, the chart restricted to $U_{pq}.$ They are required to satisfy:

\begin{enumerate}[label = (\roman*)]
\item $\Phi_{pp} = \mathrm{id}_{\mathcal{U}_p},$
\item $\psi_q \circ \phi_{pq} = \psi_p$ on $s_{p}^{-1}(0) \cap U_{pq},$
\item $\phi_{qr} \circ \phi_{pq} = \phi_{pr}$ on $s_p^{-1}(0) \cap \phi_{qr}^{-1}(U_{pq}) \cap U_{pr},$
\item $\psi_p\big(s_p^{-1}(0) \cap U_{pq}\big) =\mathrm{Im} \psi_p \cap\mathrm{Im} \psi_q,$
\end{enumerate}
\end{defn}

\begin{rem}
Here the cocycle condition is imposed {on the base maps alone,} and not on the $L_{\infty}[1]$-component. The reason for this is that $L_{\infty}$-compatibilities always hold. 
\end{rem}

\begin{defn}
A pair of the compact topological space $X$ and a collection of Kuranishi charts with coordinate changes
\[
\left(X, \widehat{\mathcal{U}}\right),
\]
where $ \widehat{\mathcal{U}}=\left(\left\{\widehat{\mathcal{U}}_p\right\}, \left\{\Phi_{pq}\right\}\right),$ is called an $L_{\infty}$\textit{-Kuranishi atlas}. For technical reasons, we assume that $\max\limits_{p \in X} \dim U_p < \infty$ with the compactness of $X$.
\end{defn}

\subsection{Definition of $L_{\infty}$-Kuranishi spaces and their categorical structures}

By considering equivalence relations on atlases, we define $L_{\infty}$-Kuranishi spaces. We also define morphisms between them, obtaining the category of $L_{\infty}$-Kuranishi spaces.

Two atlases are said to be equivalent, denoted $\left(X, \widehat{\mathcal{U}}\right) \sim \left(X, \widehat{\mathcal{U}}'\right)$, or simply $\widehat{\mathcal{U}} \sim \widehat{\mathcal{U}}'$, if we have an \textit{equality between the expanded atlases}
\[
\widehat{\mathcal{U}}^0_1 \times \mathbb{R}^{m_1} = \widehat{\mathcal{U}}^0_2 \times \mathbb{R}^{m_2},
\]
for some $m_1, m_2 \geq 0$, where $\widehat{\mathcal{U}}^0_i$ is an open subcharts of $\widehat{\mathcal{U}}_i, \ i = 0,1.$ Since we are mostly interested in the abstract categorical structure, we omit the full details regarding them. Their precise definitions can be found in \cite[Definition 3.7]{Kim1}. By \cite[Lemma 3.8 (iii)]{Kim1}, this defines an equivalence relation.

\begin{defn}[$L_{\infty}$-Kuranishi spaces]
We define an $L_{\infty}$-\textit{Kuranishi space} to be an equivalence class with respect to the relation $\sim$
\[
\mathfrak{X} := \left(X, \left[\widehat{\mathcal{U}}\right]\right).
\]
\end{defn}

Let $\mathfrak{X} = \left(X, \left[\widehat{\underline{\mathcal{U}}}\right]\right)$ and $\mathfrak{Y} = \left(Y,\left[\underline{\widehat{\mathcal{U}}}'\right]\right)$ be $L_{\infty}$-Kuranishi spaces.
We consider two atlases $\left(X, \widehat{\mathcal{U}}\right)$ and $\left(X', \widehat{\mathcal{U}}'\right)$ such that $\widehat{\mathcal{U}} \sim \underline{\widehat{\mathcal{U}}}$ and $\widehat{\mathcal{U}}' \sim \underline{\widehat{\mathcal{U}}}'$ with $\widehat{\mathcal{U}} = \left(\{\widehat{\mathcal{U}}_p\}, \left\{\Phi_{pq}\right\}\right) $ and $\widehat{\mathcal{U}}' = \left(\{\widehat{\mathcal{U}}'_{p'}\}, \left\{\Phi'_{p'q'}\right\}\right).$

\begin{defn}[Pre-morphisms]\label{morphismkur}
A \textit{pre-morphism} is defined by the following tuple
\[
\overline{F} := \left(\widehat{\mathcal{U}}, \widehat{\mathcal{U}}', f, \{f_p\}, \left\{\widehat{f}_{p,x}\right\}\right).
\]
$f : X \rightarrow Y$ is a continuous map between the zero loci, while $\left\{\left(f_p, \left\{\widehat{f}_{p,x}\right\}\right)\right\}$ is a collection of chart morphisms. Then $\overline{F}$ is required to satisfy the following compatibilities with respect to the coordinate change $\Phi_{pq}=\left(\phi_{pq}, \left\{\widehat{\phi}_{pq,x}\right\}\right)$:
For $p,q \in X$ with Im$\psi_p \cap\mathrm{Im} \psi_q \neq \emptyset,$ we require
\begin{enumerate}[label = (\roman*)]
\item $\phi'_{f(p)f(q)} \circ f_p = f_q \circ \phi_{pq}$ on $s_{p}^{-1}(0) \cap U_{pq},$
\item $\widehat{\phi}_{pq,x} \circ \widehat{f}_{q, \phi_{pq}(x)} = \widehat{f}_{p,x} \circ \widehat{\phi}'_{f(p)f(q),f_p(x)}$ for each $x \in s_p^{-1}(0) \cap U_{pq}$ up to $L_{\infty}[1]$-homotopy.
\end{enumerate}
\end{defn}

We can define an equivalence relation on the set of pre-morphisms as in \cite[Definition 3.12]{Kim1} and \cite[Lemma 3.13]{Kim1}. 

\begin{defn}[Morphism of Kuranishi spaces]
We define a \textit{morphism} from $\mathfrak{X} = \left(X, \left[\widehat{\mathcal{U}}\right]\right)$ to $\mathfrak{X}' = \left(X', \left[\widehat{\mathcal{U}}'\right]\right)$
by an equivalence class of a pre-morphism $\overline{F}$ from $\mathfrak{X}$ to $\mathfrak{X}':$
\[F := \left[\overline{F}\right] : \mathfrak{X} \rightarrow \mathfrak{X}'.\]
\end{defn}

\begin{defn}[Composition of morphisms]
Let $\mathfrak{X} = \left(X, \left[\widehat{\mathcal{U}}\right]\right),$ $\mathfrak{X}' = \left(X', \left[\widehat{\mathcal{U}'}\right]\right),$ and $\mathfrak{X}'' = \left(X'', \left[\widehat{\mathcal{U}''}\right]\right)$ be Kuranishi spaces. Let $F : \mathfrak{X} \rightarrow \mathfrak{X}'$ and $G : \mathfrak{X}' \rightarrow \mathfrak{X}''$ be morphisms between them represented by 
\begin{equation}\nonumber
\begin{cases}
\overline{F} = \left(\widehat{\mathcal{U}}, \widehat{\mathcal{U}'}, f, \left\{f_{p}\right\}, \left\{\widehat{f}_{p,x}\right\}\right),\\
\overline{G} = \left(\underline{\widehat{\mathcal{U}'}}, \widehat{\mathcal{U}''}, g, \left\{g_{f(p)}\right\}, \left\{\widehat{g}_{f(p),y}\right\}\right),
\end{cases}
\end{equation}
respectively with $\left[ \widehat{\mathcal{U}'}\right] = \left[\underline{\widehat{\mathcal{U}'}}\right]$.

In fact, we may assume that ${\widehat{\mathcal{U}'}} = \underline{\widehat{\mathcal{U}'}}$ and that $f_{p}$ is surjecive for each $p$; if not, we take the \textit{extension} of the pre-morphisms having equivalent charts as components. (See \cite[(3.6) \& Definition 3.15]{Kim1} for its precise definition). We then define the \textit{composition} $G \circ F$ to be the following equivalence class:
\begin{equation}\label{morcom}
G \circ F := \left[\left(\widehat{\mathcal{U}}, \widehat{\mathcal{U}}'', g\circ f, \left\{{g}_{f(p)} \circ {f}_{p}\right\}, \left\{{\widehat{f}}_{p,x} \circ {\widehat{g}}_{f(p),f_p(x)}\right\} \right)\right].
\end{equation}
\end{defn}

\begin{prop}\cite[Proposition 3.16]{Kim1}\label{ppidx}
The composition is well-defined and associative with the identity given by
\begin{equation}\label{idxm}
\mathrm{id}_{\mathfrak{X}}:= \left[\left(\widehat{\mathcal{U}}, \widehat{\mathcal{U}}, \mathrm{id}_X, \left\{\mathrm{id}_p\right\}, \left\{\widehat{\mathrm{id}}_{p,x}\right\}\right)\right]
\end{equation}
of each $\mathfrak{X} = \left(X, \left[\widehat{\mathcal{U}}\right]\right).$
\end{prop}

The above data give rise to a category denoted by $\mathbf{Kur}$ that consists of:
\begin{equation}\nonumber
\begin{cases}
\text{Ob}(\textbf{Kur}) =\{L_{\infty}\text{-Kuranishi spaces}\}\\
\text{Mor}(\textbf{Kur}) = \{\text{Equivalence classes of pre-morphisms} \}.
\end{cases}
\end{equation}

Indeed, $\mathbf{Kur}$ contains the category of smooth manifolds as a subcategory, allowing us to treat Kuranishi spaces and smooth manifolds on equal footing.
\begin{thm}\cite[Propositions 3.16 \& 3.18]{Kim1}
$L_{\infty}$-Kuranishi spaces form a category that admits a natural embedding from the category of smooth manifolds.
\end{thm}

\section{Hypercoverings for $L_\infty$-Kuranishi atlases}\label{hypks}

In Sections 3 and 4, we address the question raised in the introduction as to why the $L_{\infty}$-compatibilities in the definition of a Kuranishi atlas are not required. To this end, we work with \textit{hypercoverings} rather than \v{C}ech coverings (cf.\ \cite{DHI}). In preparation for Section 4, we introduce these in the present section.

\subsection{Simplicial set $N\left({\widehat{\mathcal{U}}}\right)_{\bullet}$}
We propose a method for incorporating simplicial structures into a Kuranishi atlas. Let $\left(X, \widehat{\mathcal{U}}\right)$ be a Kuranishi atlas. We associate to it a family of sets $N(\widehat{\mathcal{U}})_{\bullet}$, defined as follows:
\begin{itemize}
\item[--] $N\left(\widehat{\mathcal{U}}\right)_0 := X,$
\item[--] $N\left(\widehat{\mathcal{U}}\right)_1 := \left\{\alpha := (\alpha_0, \alpha_1) \in N\left(\widehat{\mathcal{U}}\right)_0^{\times 2} \mid \mathrm{Im}  \psi_{\alpha_0} \cap \mathrm{Im} \psi_{\alpha_1} \neq \emptyset, \ \exists \text{coord. change } \Phi_{\alpha} \right\},$
\item[--] $N\left(\widehat{\mathcal{U}}\right)_2 := \left\{ \alpha := (\alpha_0, \alpha_1, \alpha_2) \in N\left(\widehat{\mathcal{U}}\right)_1^{\times 3} \mid \partial_{t-1} \alpha_s = \partial_s \alpha_t, 0 \leq s < t \leq 2 \right\},$
\item[--]
\begin{equation}\nonumber
\begin{split}
N\left(\widehat{\mathcal{U}}\right)_{k \geq 3} := \Bigl\{ \alpha = (\alpha_0, \alpha_1, \cdots, \alpha_k) & \in N\left(\widehat{\mathcal{U}}\right)_{k-1}^{\times k+1}\\
& \mid \partial_{t-1} \alpha_s = \partial_s \alpha_t, 0 \leq s < t \leq k \Bigr\}.
\end{split}
\end{equation}
\end{itemize}

For $\alpha \in N\left(\widehat{\mathcal{U}}\right)_{\bullet},$ we denote by $v_i := v_i(\alpha)$ its $i$-th vertex. Here, $\partial_i$ denotes the face map that extracts the $i$-th component.

We denote
\begin{equation}\label{ualpha}
U_{\alpha} := \bigcap\limits_{\substack{v_l(\beta_i) = v_k(\alpha),\\ \beta_i \in N(\widehat{\mathcal{U}})_l,\\ 0 \leq l <k}} \phi^{-1}_{v_0(\alpha)v_0(\beta_i)}(U_{\beta_i}) \subset U_{v_0(\alpha)}.
\end{equation}

For low degrees, the sets $U_{\alpha}$ are given, for example, as follows:
\begin{enumerate}
\item $|\alpha| = 0,$ $U_{\alpha} = U_{p} \text{ for some } p \in X.$
\item $|\alpha| = 1,$ $U_{\alpha} \subset U_{\alpha_0}$ is the open subset associated with the coordinate change, satisfying
\begin{equation}\label{indhyp}
\psi_{\alpha_0}\big( U_{\alpha} \cap s_{\alpha_0}^{-1}(0) \big) = \mathrm{Im}  \psi_{\alpha_0} \cap \mathrm{Im} \psi_{\alpha_1},
\end{equation}
which follows from the definition.
\item $|\alpha| = 2,$ $U_{\alpha}:= \phi^{-1}_{v_0(\alpha)v_1(\alpha)}(U_{\alpha_2}) \cap U_{\alpha_1}$ is an open subset of $U_{v_0(\alpha)}.$ Observe that $U_{\alpha_2} \subset U_{v_1(\alpha)}$ and $U_{\alpha_1} \subset U_{v_0(\alpha)}.$ We can also verify that
\begin{equation}\nonumber
\psi_{v_0(\alpha)}\big( U_{\alpha} \cap s_{v_0(\alpha)}^{-1}(0) \big) = \mathrm{Im}  \psi_{v_0(\alpha)} \cap \mathrm{Im} \psi_{v_1(\alpha)} \cap \mathrm{Im}  \psi_{v_2(\alpha)}
\end{equation}
(cf. Lemma \ref{llpi}).
\end{enumerate}

\begin{assump}\label{contrassmp}
For our discussion of hypercoverings, we assume that all $U_{\alpha}$'s are \textit{contractible} open subsets, indexed by the simplices $\alpha$ of the simplicial set $N\left(\widehat{\mathcal{U}}\right)_{\bullet}.$
\end{assump}

\begin{lem}\label{llpi}
We have
\begin{equation}\nonumber
\psi_{v_0(\alpha)} \big(U_{\alpha} \cap s^{-1}_{v_0(\alpha)}(0)\big) = \bigcap\limits_{i = 0, \cdots, k} \mathrm{Im} \psi_{v_i{(\alpha)}}.
\end{equation}
\end{lem}

\begin{proof}
Since $s^{-1}_{v_0(\alpha)}(0) \subset \phi^{-1}_{v_0(\alpha)v_0(\beta_i)} \big( s^{-1}_{v_0(\beta_i)}(0) \big)$ for all $\beta_i,$ we obtain
\begin{equation}\nonumber
s^{-1}_{v_0(\alpha)}(0) = s^{-1}_{v_0(\alpha)}(0) \cap \bigcap\limits_{\substack{v_l(\beta_i) = v_k(\alpha),\\ \beta_i \in N(\widehat{\mathcal{U}})_l, \\ 0 \leq l <k}} \phi^{-1}_{v_0(\alpha)v_0(\beta_i)} \big( s^{-1}_{v_0(\beta_i)}(0) \big).
\end{equation}
(In particular, for $\beta'_i$ with $v_0(\beta'_i) = v_0(\alpha),$ we have $s^{-1}_{v_0(\alpha)}(0) = \phi^{-1}_{v_0(\alpha)v_0(\beta'_i)} \big( s^{-1}_{v_0(\beta'_i)}(0) \big).$)
Hence, we obtain
\beastar
\psi_{v_0(\alpha)} \big(U_{\alpha} \cap s^{-1}_{v_0(\alpha)}(0)\big)
& = & \psi_{v_0(\alpha)} \big(U_{\alpha} \cap s^{-1}_{v_0(\alpha)}(0) \cap \bigcap\limits_{(\cdots)} \phi^{-1}_{v_0(\alpha)v_0(\beta_i)} \big( s^{-1}_{v_0(\beta_i)}(0) \big) \\
&\overset{(1)}{=} & \psi_{v_0(\alpha)} \big( \bigcap\limits_{(\cdots)} \phi^{-1}_{v_0(\alpha)v_0(\beta_i)} \big( U_{\beta_i} \cap s^{-1}_{v_0(\beta_i)}(0) \big) \big)\\
&\overset{(2)}{=} & \bigcap\limits_{(\cdots)} \psi_{v_0(\alpha)} \circ \phi^{-1}_{v_0(\alpha)v_0(\beta_i)}
\big( U_{\beta_i} \cap s^{-1}_{v_0(\beta_i)}(0) \big) \\
&= &  \bigcap\limits_{(\cdots)} \psi_{v_0(\beta_i)} \big( U_{\beta_i} \cap s^{-1}_{v_0(\beta_i)}(0) \big) \\
& \overset{(3)}{=} & \bigcap\limits_{(\cdots)} \bigcap\limits_{j = 0, \dots, |\beta_i|} \mathrm{Im} \psi_{v_j(\beta_i)}
= \bigcap\limits_{i = 0, \cdots, k} \mathrm{Im} \psi_{v_i(\alpha)},
\eeastar
where (1) follows from (\ref{ualpha}), i.e., the definition of $U_{\alpha},$ while equality $(2)$ follows from the injectivity of the map $\psi_{v_0(\alpha)}$, and equality $(3)$ follows from the induction hypothesis (\ref{indhyp}).
\end{proof}

We consider the face maps
$$
\partial_i : N\left(\widehat{\mathcal{U}}\right)_k \rightarrow N\left(\widehat{\mathcal{U}}\right)_{k-1}, \ i = 0, \cdots, k,\quad
\partial_i(\alpha_0, \cdots, \alpha_k) : = \alpha_i,
$$
and the degeneracy maps
$$
\sigma_i : N\left(\widehat{\mathcal{U}}\right)_k \rightarrow N\left(\widehat{\mathcal{U}}\right)_{k+1}, \ i = 0, \cdots, k,
$$
defined by
$$
\alpha = (\alpha_0, \cdots, \alpha_k) \mapsto
\begin{cases}
(\alpha, \alpha, \sigma_0 \alpha_1, \cdots, \sigma_0 \alpha_k) & \text{ if } i = 0,\\
(\sigma_{i-1} \alpha_0, \cdots, \sigma_{i-1} \alpha_{i-1}, \alpha, \alpha, \sigma_{i} \alpha_{i+1}, \cdots, \sigma_{i} \alpha_k) & \text{ if } 1 \leq i \leq k-2,\\
(\sigma_{k-1} \alpha_0, \cdots, \sigma_{k-1} \alpha_{k-1}, \alpha, \alpha) & \text{ if } i =k-1,\\
(\alpha, \sigma_{k} \alpha_0, \cdots, \sigma_{k} \alpha_{k}, \alpha) & \text{ if } i =k.\\
\end{cases}
$$

\begin{lem}\label{lemu} The following properties hold:
\begin{enumerate}[label = (\roman*)]
\item The degeneracy maps are well-defined.
\item $v_0(\alpha) = v_0(\sigma_j \alpha)$ and $v_k(\alpha) = v_{k+1}(\sigma_j \alpha)$ for each $0 \leq j \leq k.$
\item $U_{\alpha} = U_{\sigma_j \alpha}$ for each $0 \leq j \leq k.$
\item $U_{\alpha} \subset U_{\partial_j \alpha},$ for each $0 \leq j < k.$
\item $U_{\alpha} \subset \phi_{v_0(\alpha)v_1(\alpha)}^{-1} (U_{\partial_k \alpha}).$
\item $\left(N\left(\widehat{\mathcal{U}}\right)_{\bullet}, \{\partial_j\}, \{\sigma_j\}\right)$ is a simplicial set.
\end{enumerate}
\end{lem}

\begin{proof} We prove each statement in turn:
\begin{enumerate}[label = (\roman*)]
\item This follows from a straightforward computation verifying that $\partial_{t-1} \alpha_s = \partial_s \alpha_t$ holds for all $0 \leq s < t \leq k.$
\item The $0$-th vertex and the $k$-th vertex of the $k$-simplex $\alpha$ are characterized by the repeated compositions of face maps $\partial_0 \circ \cdots \circ \partial_0(\alpha)$ and $\partial_1 \circ \cdots \circ \partial_k(\alpha),$ respectively. One checks that $\partial_0 \circ \cdots \circ \partial_0(\alpha) = \partial_0 \circ \cdots \circ \partial_0(\sigma_j \alpha)$ and $\partial_1 \circ \cdots \circ \partial_k(\alpha) = \partial_1 \circ \cdots \circ \partial_{k+1}(\sigma_j \alpha)$ for all $j.$
\item We proceed by induction. It is straightforward to verify the statement for $k=2.$ By (ii), the additional indexing simplex $\beta_j'$ in (\ref{ualpha}) needed for $\sigma_j \alpha$ (as compared with $\alpha$) is either $\alpha$ itself, or a degenerate simplex of smaller degree $\leq k.$ In the former case, we have $\phi^{-1}_{v_0(\alpha) v_0(\alpha)} (U_{\alpha})= U_{\alpha}.$ In the latter case, for $\beta'_j$ with $|\beta'_j| \leq k-1,$ we have $v_0(\sigma_j \beta'_j) = v_0(\beta'_j)$ and $\phi^{-1}_{v_0(\alpha)v_0(\sigma_j \beta'_j)} (U_{\sigma_j \beta'_j}) = \phi^{-1}_{v_0(\alpha)v_0(\beta'_j)} (U_{\beta'_j})$ by the induction hypothesis. Taking the intersection of all these components, we obtain $U_{\alpha} = U_{\sigma_j \alpha}.$
\item We have $v_0(\partial_j \alpha) = v_0(\alpha)$ and $v_{k-1}(\partial_j \alpha) = v_k(\alpha),$ so
\beastar
U_{\alpha} & = &  \bigcap\limits_{\substack{v_l(\beta_i)
= v_k(\alpha),\\ \beta_i \in  N(\widehat{\mathcal{U}})_l,\\ 0 \leq l <k}} \phi^{-1}_{v_0(\alpha)v_0(\beta_i)}(U_{\beta_i}) \subset \bigcap\limits_{\substack{v_l(\beta_i) = v_k(\alpha),\\ \beta_i \in N(\widehat{\mathcal{U}})_l,\\ 0 \leq l <k-1}} \phi^{-1}_{v_0(\alpha)v_0(\beta_i)}(U_{\beta_i}) \\
&= & \bigcap\limits_{\substack{v_l(\beta_i) = v_{k-1}(\partial_j \alpha),\\ \beta_i \in  N(\widehat{\mathcal{U}})_l,\\ 0 \leq l <k-1}} \phi^{-1}_{v_0(\partial_j \alpha)v_0(\beta_i)}(U_{\beta_i}) = U_{\partial_j \alpha}.
\eeastar
\item We have $v_0(\partial_k \alpha) = v_1(\alpha)$ and $v_{k-1}(\partial_k \alpha) = v_k(\alpha),$ hence
\beastar
U_{\alpha} & = & \bigcap\limits_{\substack{v_l(\beta_i) = v_k(\alpha),\\ \beta_i \in  N(\widehat{\mathcal{U}})_l,\\ 0 \leq l <k}} \phi^{-1}_{v_0(\alpha)v_0(\beta_i)}(U_{\beta_i}) \subset \bigcap\limits_{\substack{v_l(\beta_i) = v_k(\alpha),\\ \beta_i \in  N(\widehat{\mathcal{U}})_l,\\ 0 \leq l <k-1}} \phi^{-1}_{v_0(\alpha)v_0(\beta_i)}(U_{\beta_i})\\
&= &  \bigcap\limits_{\substack{v_l(\beta_i) = v_k(\alpha),\\ \beta_i \in  N(\widehat{\mathcal{U}})_l,\\ 0 \leq l <k-1}} \phi^{-1}_{v_0(\alpha)v_1 (\alpha)} \circ \phi^{-1}_{v_1(\alpha)v_0(\beta_i)}(U_{\beta_i})\\
& \overset{*}{=} & \phi^{-1}_{v_0(\alpha)v_1(\alpha)} \left( \bigcap\limits_{\substack{v_l(\beta_i)
= v_{k-1}(\partial_k \alpha),\\ \beta_i \in  N(\widehat{\mathcal{U}})_l,\\ 0 \leq l <k-1}}
\phi^{-1}_{v_0(\partial_k \alpha)v_0(\beta_i)}(U_{\beta_i}) \right)\\
& = &  \phi^{-1}_{v_0(\alpha)v_1(\alpha)}(U_{\partial_k \alpha}),
\eeastar
where equality $*$ holds because $\phi^{-1}_{v_0(\alpha)v_1(\alpha)}$ is injective.
\item The simplicial identities can be verified in a routine manner, and we omit the details.
\end{enumerate}
\end{proof}

\subsection{Kuranishi hypercoverings}
The simplicial set $N\left(\widehat{\mathcal{U}}\right)_{\bullet}$ introduced in the previous subsection serves as a family of parameters for systematically covering the underlying topological space. This is precisely the role played by Kuranishi hypercoverings.

\begin{defn}\label{khdef}
Let $X$ be a topological space. Given a simplicial set $S_{\bullet}$, we consider a family of subsets $\{V_{\alpha}\}_{\alpha \in S_{\bullet}}$ of $X$ indexed by the simplices of $S_{\bullet}$. We call this a \textit{hypercovering} of $X$ if the following hold:
\begin{enumerate}[label=(\roman*)]
\item $\bigcup\limits_{\alpha \in S_0} V_{\alpha} = X$,
\item $V_{\partial_i \alpha} \supset V_{\alpha}$,
\item $V_{\sigma_i \alpha} = V_{\alpha}$,
\item $V_{\alpha_0} \cap V_{\alpha_1} = \bigcup\limits_{\substack{ \alpha \in S_1, \\ \partial_i \alpha = \alpha_i, \ i = 1,2}} V_{\alpha}$,
\item $\bigcap\limits_{i = 0, \cdots, k} V_{\alpha_i} = \bigcup\limits_{\substack{ \alpha \in S_k, \\ \partial_i \alpha = \alpha_i }} V_{\alpha}$ for all $\alpha_1, \cdots, \alpha_k \in S_k$ with $\partial_{t-1} \alpha_s = \partial_s \alpha_t, \ 0 \leq s < t \leq k,$ and $2 \leq k$.
\end{enumerate}
\end{defn}

\begin{exam}[Hypercovering induced by $N\left(\widehat{\mathcal{U}}\right)_{\bullet}$]
Given a Kuranishi atlas $(X, \widehat{\mathcal{U}})$, let $\alpha$ be a $k$-simplex of $N\left(\widehat{\mathcal{U}}\right)_{\bullet}$, as in the previous subsection. Consider the family of subsets
\begin{equation}\nonumber
V_{\alpha} := \psi_{v_0(\alpha)}\big( U_{\alpha} \cap s_{v_0(\alpha)}^{-1}(0)\big) \subset V_{v_0(\alpha)}.
\end{equation}
Note that $V_{\alpha} = \bigcap\limits_{i = 0, \cdots, k} \mathrm{Im}\,\psi_{v_i(\alpha)}$ by Lemma \ref{llpi}.
\end{exam}

\begin{prop}\label{kurhyp}
$\{V_{\alpha}\}$ is a hypercovering of $X$.
\end{prop}
\begin{proof}
We verify conditions (i) through (v) of Definition \ref{khdef}. Condition (i) follows from the definition of a Kuranishi atlas.
Condition (ii) follows from
$$
\big\{v_0(\partial_i \alpha), \cdots, v_k(\partial_i \alpha)\big\} \subset \big\{v_0(\alpha), \cdots, v_k(\alpha)\big\},
$$
and (iii) from
$$
\big\{v_0(\sigma_i \alpha), \cdots, v_k(\sigma_i \alpha)\big\} = \big\{v_0(\alpha), \cdots, v_k(\alpha)\big\}.
$$
Condition (iv) follows from Definition \ref{kstr} (iv). We remark that all the $V_{\alpha}$ on the right-hand side of (iv) coincide; that is, $V_{\alpha}$ for a $1$-simplex $\alpha$ is independent of such choices, even though different $\alpha \in N(\widehat{\mathcal{U}})_k$ may occur. The same observation applies to our proof of condition (v). We prove (v) by induction, assuming the analogous equality (the induction hypothesis) holds for each $V_{\alpha_i}$, namely,
\begin{equation}\nonumber
V_{\alpha_i} = \bigcap\limits_{j = 0, \cdots, k-1}V_{\partial_j \alpha_i}
= \bigcap\limits_{j = 0, \cdots, k-1}\mathrm{Im}\,\psi_{v_j(\partial_j \alpha_i)}.
\end{equation}
Then
\begin{equation}\nonumber
\bigcup\limits_{\substack{i = 0, \cdots, k, \\ j = 0, \cdots, k-1}} \big\{v_j(\alpha_i)\big\}
= \big\{v_0(\alpha), \cdots, v_k(\alpha)\big\}
\end{equation}
implies the desired equality.
\end{proof}

\subsection{Kuranishi chart category}\label{ksctr}

In this subsection, we prepare for a rigorous definition of higher cocycle conditions in Section 4 by introducing a simplicially enriched category associated to a Kuranishi space, whose objects are Kuranishi charts that combine to form an atlas equivalent to the given one. This construction provides the foundation for a precise formulation of higher cocycle conditions.

Given a Kuranishi space $\mathfrak{X}= \left(X, \left[\widehat{\mathcal{U}}\right]\right),$ we define the \emph{Kuranishi chart category} $\mathcal{K}_{\mathfrak{X}}$ as follows.

Let $\{U_p\}_{p \in X}$ be a collection of open subsets of $X,$ and let
$U_\alpha$ be as in (\ref{ualpha}). The objects of $\mathcal{K}_{\mathfrak{X}}$ are given by
\begin{equation}\nonumber
\text{Ob}(\mathcal{K}_{\mathfrak{X}}) := \left\{\mathcal{U}_p  \mid p \in X, \ \mathcal{U}_p \in {\widehat{\mathcal{U}}}_1 \text{ with } \left[{\widehat{\mathcal{U}}}_1\right] = \left[\widehat{\mathcal{U}}\right]\right\}.
\end{equation}

For a pair $\mathcal{U}_p, \underline{\mathcal{U}}_q \in \text{Ob}(\mathcal{K}_{\mathfrak{X}}),$ we consider the set
\begin{equation}\label{spq}
\begin{split}
\mathscr{S}(\mathcal{U}_p, \underline{\mathcal{U}}_q) := \begin{cases} \big\{ \mathscr{W} \subset U_p
\mid \mathscr{W} \text{ is a contractible open subset satisfying }\\
\quad \quad \quad \quad \quad \quad \quad \quad \quad \quad \quad \quad \psi_p(s^{-1}_p(0) \cap \mathscr{W}) \subset \underline{\psi}_q\big(\underline{s}_q^{-1}(0)\big)\big\},\\
\emptyset  \quad \quad \quad \quad \text{ if no such $\mathscr{W}$ exists}.
\end{cases}
\end{split}
\end{equation}

For $p, q$ with $V_p \cap V_q \neq \emptyset,$ the morphism space is given by
\begin{equation}\nonumber
\text{Mor}_{\mathcal{K}_{\mathfrak{X}}}(\mathcal{U}_p, \underline{\mathcal{U}}_q):= \coprod\limits_{k=0}^{\infty} M_{pq}^k,
\end{equation}
where $M_{pq}^k$ is given by
\begin{equation}\label{mpqk}
M_{pq}^k := \begin{cases}
\big\{ (\mathscr{W}_{pq}, \Phi_{pq}^k, \{\widehat{\Phi}_{pq,x}^{k}\}) \mid  \text{(a)--(c) below hold} \big\} &\text{ if } \dim U_p = \dim \underline{U}_q,\\
\emptyset &\text{ otherwise.}
\end{cases}
\end{equation}
\begin{enumerate}
\item[(a)] $\mathscr{W}_{pq} \in \mathscr{S}(\mathcal{U}_p, \underline{\mathcal{U}}_q)$ is an open subset as in (\ref{spq}).
\item[(b)]
\[
\Phi_{pq}^k : \mathscr{W}_{pq} \times \Delta^k \rightarrow \underline{U}_q
\]
is a smooth map satisfying:
\begin{enumerate}[label = (\roman*)]
\item $\Phi^0_{pq}$ is an embedding.
\item $\Phi_{pq}^k((s^{-1}_p(0) \cap \mathscr{W}_{pq}) \times \Delta^k) \subset \underline{s}^{-1}_q(0)$.
\item $ \underline{\psi}_q  \circ \Phi^0_{pq} = \psi_p$ on $s_p^{-1}(0) \cap \mathscr{W}_{pq}.$
\item $\Phi^k_{pq}$ restricts to a \textit{surjection} $\Phi_{pq}^k|_{W_x \times \Delta^k} : W_x \times \Delta^k \twoheadrightarrow W'_{\Phi_{pq}^0(x)}$ for each $x \in s^{-1}_p(0) \cap \mathscr{W}_{pq}$ and $k \geq 0.$
\item $\Phi_{pq}^{k} \circ (\mathrm{id}_{\mathscr{W}_{pq}} \times d_i) = \Phi_{pq}^{k-1}, \text{ for all } i.$
\end{enumerate}
\item[(c)]
\[
\widehat{\Phi}_{pq,x}^{k} :  \mathcal{C}'_{q,\Phi_{pq}^0(x)} (=\mathcal{C}'_{q,\Phi_{pq}^0(x), \Phi_{pq}^k})  \rightarrow \Omega^* (\Delta^k) \otimes \mathcal{C}_{p,x}\]
is an $L_{\infty}[1]$-morphism, and an $L_{\infty}[1]$-$k$-homotopy in the sense of Example \ref{hhex}, for each $x \in s^{-1}_p(0) \cap \mathscr{W}_{pq},$ satisfying: $\widehat{\Phi}^0_{pq,x}$ is a quasi-isomorphism for each $x \in s^{-1}_p(0) \cap \mathscr{W}_{pq}.$
\end{enumerate}
Here, $\Phi_{pq}^0$ should not be confused with the coordinate changes of Kuranishi spaces.

We next define the composition of morphisms,
\begin{equation}\nonumber
\circ : \text{Mor}_{\mathcal{K}_{\mathfrak{X}}}\big( \underline{\mathcal{U}}_q,  \underline{\mathcal{U}}'_r \big) \times \text{Mor}_{\mathcal{K}_{\mathfrak{X}}}\big(\mathcal{U}_p, \underline{\mathcal{U}}_q \big)  \rightarrow \text{Mor}_{\mathcal{K}_{\mathfrak{X}}}\big(\mathcal{U}_p, \underline{\mathcal{U}}'_r \big),
\end{equation}
by specifying the composition
\begin{equation}\nonumber
\begin{split}
\circ : M_{qr}^{k} \times M_{pq}^k &\rightarrow M_{pr}^k.\\
\end{split}
\end{equation}
For each composable pair
\[
\big( (\mathscr{W}_{qr}, \Phi_{qr}^k, \widehat{\Phi}_{qr}^{k}),  (\mathscr{W}_{pq}, \Phi_{pq}^k, \widehat{\Phi}_{pq}^{k}) \big) \in  M_{qr}^{k} \times M_{pq}^k,
\]
we set
\begin{equation}\label{eqcb}
\begin{split}
(\mathscr{W}_{qr}, \Phi_{qr}^k, \widehat{\Phi}_{qr}^{k}) \circ & (\mathscr{W}_{pq}, \Phi^k_{pq}, \widehat{\Phi}_{pq}^{k})\\
& := \big( \mathscr{W}_{pqr}, \Phi^k_{qr} \circ (\Phi^k_{pq}|_{\mathscr{W}_{pqr}}),  \widehat{\Phi}_{qr}^{k} \circ (\widehat{\Phi}_{pq}^{k}|_{\mathscr{W}_{pqr}})\big).
\end{split}
\end{equation}
We now explain the meaning of each argument on the right-hand side of (\ref{eqcb}).

For the base, we define
\begin{equation}\nonumber
\mathscr{W}_{pqr} : = (\Phi_{pq}^{0})^{-1}(\mathscr{W}_{qr}) \cap \mathscr{W}_{pq},\\
\end{equation}
which lies in $ \mathscr{S}(\mathcal{U}_p, {\underline{\mathcal{U}}}'_r)$ by Assumption \ref{contrassmp} on contractibility. Indeed,
\begin{equation}\nonumber
\begin{split}
 \psi_p\left(s^{-1}_p(0) \cap \mathscr{W}_{pqr}\right) & = \psi_p\left(s^{-1}_p(0) \cap (\Phi_{pq}^{0})^{-1}(\mathscr{W}_{qr}) \cap \mathscr{W}_{pq}\right) \\
&= \psi_p\left( (\Phi_{pq}^{0})^{-1}(\underline{s}^{-1}_q(0) \cap \mathscr{W}_{qr}) \cap s^{-1}_p(0) \cap \mathscr{W}_{pq}\right)\\
&= \psi_p\left( (\Phi_{pq}^{0})^{-1}(\underline{s}^{-1}_q(0) \cap \mathscr{W}_{qr})\right) \cap \psi_p\left(s^{-1}_p(0) \cap \mathscr{W}_{pq}\right)\\
&= \underline{\psi}_q\left(\underline{s}^{-1}_q(0) \cap \mathscr{W}_{qr}) \cap \psi_p(s^{-1}_p(0) \cap \mathscr{W}_{pq}\right)\\
&\subset \underline{\psi}_q\left(\underline{s}_q^{-1}(0)\right),
\end{split}
\end{equation}
using axioms (i) through (iii) of (b) above, together with the homeomorphism property of $\psi_p$ and $\underline{\psi}_{q}.$

For a fixed vector $\vec{t} \in \Delta^k,$ we write
\[
\Phi^k_{pq,\vec{t}} := \Phi^k_{pq}\left(\cdot, \vec{t}\right): \mathscr{W}_{pq} \rightarrow \underline{U}_q,
\]
and, fixing a basis $\left\{\gamma_i\right\}_i$ of $\Omega^*\left(\Delta^k\right)$ together with
\[
\left\{\widehat{\Phi}_{pq,x,i}^{k} : \mathcal{C}'_{q,\Phi_{pq}^0(x)} (=  \mathcal{C}'_{q,\Phi_{pq}^0(x), \Phi_{pq}^k}) \rightarrow \mathcal{C}_{p,x}\right\}_i,
\]
we write
\[
\widehat{\Phi}_{pq,x}^{k} = \sum\limits_i \gamma_i \otimes \widehat{\Phi}_{pq,x,i}^{k}.
\]

We then define
\begin{equation}\label{phpkc}
\begin{split}
\Phi^k_{qr} \circ \left(\Phi^k_{pq}|_{\mathscr{W}_{pqr}}\right) \left(x, \vec{t}\right) &:= \left(\left(\Phi^k_{qr,\vec{t}} \circ \Phi^k_{pq,\vec{t}}|_{\mathscr{W}_{pqr}}\right)(x), \vec{t}\right),\\
\widehat{\Phi}^k_{pq,x} \circ \left(\widehat{\Phi}^{k}_{qr, \Phi_{pq}^0(x)}  |_{\mathscr{W}_{pqr}}\right) (\xi)
&:= \sum\limits_i \gamma_i \otimes \left(\widehat{\Phi}^{k}_{pq,x,i} \circ \widehat{\Phi}^{k}_{qr, \Phi^0_{pq}(x),i}|_{\mathscr{W}_{pqr}} \right)(\xi)
\end{split}
\end{equation}
for $\left(x, \vec{t}\right) \in \mathscr{W}_{pqr} \times \Delta^k$ and
$\xi \in \mathcal{C}'_{r, \Phi_{qr}^0(x)}.$

For each object $\mathcal{U}_p,$ the identity morphism $\left(U_p, \mathrm{id}_{p}^k, \left\{\widehat{\mathrm{id}}^{k}_{p,x}\right\}\right) \in M_{pp}^k$ is defined as follows: we set
$$
\mathrm{id}_{p,x}^k : U_p \times \Delta^k \rightarrow U_p,\quad
\mathrm{id}_{p,x}^k (x,\vec{t}) = x, \ \text{ for each } \vec{t} \in \Delta^{k},
$$
and define the map
$$
\widehat{\mathrm{id}}^{k}_{p,x} : \mathcal{C}_{p,x} \rightarrow \Omega^* (\Delta^k) \otimes \mathcal{C}_{p,x}
$$
by
$$
\widehat{\mathrm{id}}^{k}_{p,x} =
\begin{cases}
0 &\text{ if } k \geq 1,\\
\widehat{\mathrm{id}}_{\mathcal{C}_{p,x}} &\text{ if } k = 0.
\end{cases}
$$
One readily checks that the composition defined above is associative, and that the identity morphism is indeed an identity for this composition.

The following lemma is immediate.
\begin{lem}\label{lss}
$\text{Mor}_{\mathcal{K}_{\mathfrak{X}}}\left( {\mathcal{U}}_{{p}},  {\mathcal{U}}_{{q}} \right)$ is a simplicial set, with face and degeneracy maps given, for each $k \geq 0,$ by
\[
\partial_i : M^k_{pq} \rightarrow  M_{pq}^{k-1}; \quad
\partial_i\left(\mathscr{W}_{pq}, \Phi_{pq}^k, \widehat{\Phi}_{pq}^k\right) =
 \left(\mathscr{W}_{pq}, \Phi_{pq}^k \circ d_i,  d^*_i \circ \widehat{\Phi}_{pq}^k\right), \ i = 0, \cdots, k,
\]
and
\[
\sigma_i : M_{pq}^k  \rightarrow  M_{pq}^{k+1}; \quad
\sigma_i\left(\mathscr{W}_{pq}, \Phi_{pq}^k, \widehat{\Phi}_{pq}^k\right) = \left(\mathscr{W}_{pq}, \Phi_{pq}^k \circ s_i, s^*_i \circ \widehat{\Phi}_{pq}^k\right), \ i = 0, \cdots, k.
\]
Here, $d_i: \Delta^{k-1} \rightarrow \Delta^{k}$ and $s_i : \Delta^{k+1} \rightarrow \Delta^{k}$ are the standard coface and codegeneracy maps on the standard simplices, and $d_i^* : \Omega^*(\Delta^{k}) \rightarrow \Omega^*(\Delta^{k-1})$ and $s_i^* : \Omega^*(\Delta^{k}) \rightarrow \Omega^*(\Delta^{k+1})$ are the induced maps on de Rham complexes. Moreover, the compositions are compatible with this simplicial structure.
\end{lem}

It follows from the preceding lemma that $\mathcal{K}_{\mathfrak{X}}$ is a simplicially enriched category. In fact, it enjoys a particularly well-behaved property, which will play a central role in the arguments of the next section.

\begin{thm}\label{kcpx}
$\mathcal{K}_{\mathfrak{X}}$ is a simplicially enriched category, and moreover a Kan-complex-enriched category. That is, the morphism space $\mathrm{Mor}_{\mathcal{K}_{\mathfrak{X}}}( \mathcal{U}_{p}, \mathcal{U}_{q})$ is a Kan complex for each pair $p,q \in X$.
\end{thm}

\begin{proof}
The first statement follows from Lemma \ref{lss}. It remains only to verify that the compositions (\ref{phpkc}) are compatible with the face and degeneracy maps, which is evident from their construction. The second statement follows from the contractibility of each $\mathscr{W}$ in (\ref{mpqk}) (for the base component), together with Proposition \ref{pphhe} (for the $L_{\infty}[1]$-component).
\end{proof}

\section{Higher homotopies and the simplicial nerve
$N(\mathcal{K}_{{\mathfrak{X}}})$}
In this section, we introduce the simplicial nerve construction $N(\CK_{\mathfrak{X}})$ associated with the chart category $\CK_{\mathfrak{X}}$ arising from a given Kuranishi space $\mathfrak{X}$. Here, the standard notion of a cocycle condition is replaced by a more relaxed version, called a \textit{higher cocycle condition}, which explicitly encodes the higher homotopy information.

\subsection{Simplicial nerve construction $N(\mathcal{K}_{\mathfrak{X}})$}
As a preliminary step toward the higher cocycle condition, we consider higher homotopies, defined by exploiting the simplicial set structure on the chart category that is induced by the simplicially enriched category $\mathcal{K}_{\mathfrak{X}}.$

\begin{defn}[$m$-homotopies for morphisms of $\mathcal{K}_{\mathfrak X}$]\label{mhtpy}
Let
\[
\Phi_i = \left\{\left(\mathscr{W}_{pq}, \Phi_{i,pq}^{k}, \left\{\widehat{\Phi}_{i, pq,x}^{k}\right\}\right)\right\}_{k \geq 0}, \ i =0, \cdots, m
\]
be morphisms of $\mathcal{K}_{\mathfrak{X}}.$ We say they are $m$\textit{-homotopic} if there exists a family of tuples
\[
\overline{\Phi}^{(m)} = \left\{\left(\mathscr{W}_{pq}, \overline{\Phi}^k_{pq}, \left\{\widehat{\overline{\Phi}}^{k}_{pq,x}\right\}\right)\right\}_{k \geq 0}
\]
consisting of:
\begin{enumerate}
\item[--]
a smooth map
\[
\overline{\Phi}_{pq}^k : \mathscr{W}_{pq} \times \Delta^k \times \Delta^m \rightarrow U_q,
\]
\item[--]
for each $x \in s^{-1}_p(0) \cap \mathscr{W}_{pq},$ an $\Omega^* (\Delta^m)$-family of $L_{\infty}[1]$-morphisms and $L_{\infty}[1]$-$k$-homotopies,
\[
\widehat{\overline{\Phi}}_{pq,x}^{k} : \mathcal{C}'_{q, \Phi_{pq}^0(x), \Phi^k_{pq}} (= \mathcal{C}'_{q,\Phi_{pq}^0(x)}) \rightarrow \Omega^* (\Delta^m) \otimes \Omega^* (\Delta^k)  \otimes \mathcal{C}_{p,x},
\]
in the sense of Example \ref{hhex}, satisfying:
\begin{enumerate}[label=(\roman*)]
\item conditions analogous to those imposed on morphisms in Definition \ref{ksctr};
\item $\left(\overline{\Phi}^k_{pq}, \left\{\widehat{\overline{\Phi}}^{k}_{pq,x}\right\}\right)\bigg|_{v_i(\Delta^m)} = \left(\Phi_{i, pq}^{k}, \left\{\widehat{{\Phi}}_{i, pq,x}^{k}\right\}\right), \ i =0, \cdots, m,$ where $v_i(\Delta^m)$ denotes the $i$-th vertex of $\Delta^m.$
\end{enumerate}
\end{enumerate}
We call $\overline{\Phi}^{(m)}$ an $m$\textit{-homotopy} of the morphisms $\Phi_0, \cdots, \Phi_m.$
\end{defn}

\begin{prop}
Given any morphisms $\Phi_0, \cdots, \Phi_m$ of $\mathcal{K}_{\mathfrak{X}}$ as in Definition \ref{mhtpy}, an $m$\textit{-homotopy} $\overline{\Phi}^{(m)}$ exists for every $m \geq 1$.
\end{prop}

\begin{proof}
This follows from Corollary \ref{anhp}, together with Assumption \ref{contrassmp} that the open subsets $U_{pq}$ are all contractible.
\end{proof}

We briefly recall the notion of simplicial nerves. Let $\mathcal{C}$ be a simplicially enriched category. The $n$-simplices of the \textit{simplicial nerve} $N(\mathcal{C})$ of $\mathcal{C}$ are determined by
\begin{equation}\nonumber
\text{Hom}_{sSet} \big(\Delta^n, N(\mathcal{C}) \big) := \text{Hom}_{Cat_{\Delta}}\big( \mathfrak{C}[\Delta^n], \mathcal{C} \big).
\end{equation}

Here, $\mathfrak{C}[\Delta^n]$ is the category with
\begin{enumerate}
\item[--] $\text{Ob}(\mathfrak{C}[\Delta^n]) = \{0, \cdots, n\},$
\item[--] $\text{Mor}_{\mathfrak{C}[\Delta^n]}(i,j) := \begin{cases}
\emptyset \quad & i > j,\\
N(P_{i,j}) \quad & i \leq j,
\end{cases}$
\end{enumerate}
where $P_{i,j}$ is the (ordinary) nerve of the partially ordered set
\begin{equation}\nonumber
P_{i,j} := \big\{ I \subset \{0, \cdots, n\} \mid \text{if } i, j \in I \text{ and }  k \in I, \text{ then } i \leq k \leq j \big\},
\end{equation}
ordered by inclusion. In fact, $N(P_{i,j}) \simeq (\Delta^1)^{j-i-1}$ (i.e., a cube) as simplicial sets (cf.\ \cite{Lurie}).

\begin{cor}
$N(\mathcal{K}_{\mathfrak{X}})$ is an $\infty$-category, where $N(\cdot)$ denotes the simplicial nerve construction.
\end{cor}
\begin{proof}
This follows from Theorem \ref{kcpx} and \cite[Proposition 1.1.5.10]{Lurie}.
\end{proof}

The simplices of the simplicial set $N(\mathcal{K}_\mathfrak{X})$ in low degrees are given as follows.
In degree $0$, $N_0(\mathcal{K}_\mathfrak{X}) = \mathrm{Ob}(\mathcal{K}_{\mathfrak{X}}).$ In degree $1$, $N_1(\mathcal{K}_\mathfrak{X}) = \coprod\limits_{\mathcal{U}_p, \underline{\mathcal{U}}_q \in \mathrm{Ob}(\mathcal{K}_{\mathfrak{X}})} \text{Mor}_{\mathcal{K}_\mathfrak{X}}(\mathcal{U}_p, \underline{\mathcal{U}}_q).$ $N_2(\mathcal{K}_\mathfrak{X})$ consists of a pair of morphisms together with a $1$-homotopy between them. $N_3(\mathcal{K}_\mathfrak{X})$ consists of a diagram of five $1$-homotopies, filled in by two $2$-homotopies. $N_{\geq 4}(\mathcal{K}_\mathfrak{X})$ is constructed inductively in a similar manner, which we omit.

\subsection{Definition of higher cocycle condition}
We now define the \textit{higher cocycle condition} for a Kuranishi atlas, motivated by \cite[Definition 4.5.3]{Tu1}.

Denote
\[
\mathcal{O}^{(l)}:= \big\{ \mathcal{U}_p \in \mathrm{Ob}(\mathcal{K}_{\mathfrak{X}}) \mid \mathrm{dim}\,U_p \leq l \big\}.
\]

We define $\mathcal{K}_{\mathfrak{X}}^{(l)}$ to be the subcategory of $\mathcal{K}_{\mathfrak{X}}$ given by
\[
\mathrm{Ob}(\mathcal{K}_{\mathfrak{X}}^{(l)}) := \mathrm{Ob}(\mathcal{K}_{\mathfrak{X}})
\]
and
\[
\mathrm{Mor}_{\mathcal{K}_{\mathfrak{X}}^{(l)}}(\mathcal{U}_p, \underline{\mathcal{U}}_q) := \begin{cases}
\mathrm{Mor}_{\mathcal{K}_{\mathfrak{X}}}(\mathcal{U}_p, \underline{\mathcal{U}}_q) &\text{ if } \mathcal{U}_p, \underline{\mathcal{U}}_q \in \mathcal{O}^{(l)},\\
\{\mathrm{id}_{\mathcal{U}_p}\} &\text{ if } \mathcal{U}_p = \underline{\mathcal{U}}_q \notin \mathcal{O}^{(l)},\\
\emptyset &\text{ otherwise.}
\end{cases}
\]

Observe that the natural inclusion of categories
\[
\mathcal{K}^{(l)}_{\mathfrak{X}} \;\hookrightarrow\; \mathcal{K}^{(l+1)}_{\mathfrak{X}}
\]
induces an embedding of simplicial sets
\[
\mathscr{I}^{(l)} : N_{\bullet}\left(\mathcal{K}^{(l)}_{\mathfrak{X}}\right) \;\hookrightarrow\; N
_{\bullet}\left(\mathcal{K}^{(l+1)}_{\mathfrak{X}}\right).
\]

\begin{defn}[Higher cocycle condition]\label{defhtpgl}
Let $\mathfrak{X}$ be an $L_{\infty}$-Kuranishi space, with an atlas $\widehat{\mathcal{U}}$ representing it, together with a hypercovering. Let $ N\left(\widehat{\mathcal{U}}\right)_{\bullet}$ be the simplicial set defined in Section \ref{hypks}. We call the existence of a family of degree-preserving maps
\begin{equation}\nonumber
\big\{ \mathscr{G}_{\bullet}^{(l)} : N\left(\widehat{\mathcal{U}}\right)_{\bullet} \rightarrow N_{\bullet}\left(\mathcal{K}_{\mathfrak{X}}^{(l)}\right) \big\}_{l \geq 1}
\end{equation}
a \textit{higher cocycle condition} of $\mathfrak{X}$ if it satisfies
\begin{enumerate}[label = (\roman*)]
\item $ \mathscr{G}^{(l)}_{m-1}(\partial_j \alpha)|_{U_{\alpha}} = \partial_j \mathscr{G}^{(l)}_m(\alpha), \ j = 0, \cdots, m,$
\item  $\mathscr{G}^{(l)}_{m+1}(\sigma_j \alpha) = \sigma_j\mathscr{G}^{(l)}_m(\alpha), \ j = 0, \cdots, m,$
\item $\mathscr{G}^{(l+1)}_\bullet = \mathscr{I}^{(l)} \circ \mathscr{G}^{(l)}_\bullet,$
\end{enumerate}
for each $l \geq 1,$ where $(\cdot)|_{U_{\alpha}}$ in (i) denotes the obvious restriction to the open subset $U_{\alpha}.$
\end{defn}

We now state a key theorem of this section.
\begin{thm}[Existence of higher cocycle conditions]\label{htpyexist}
Higher cocycle conditions exist for any Kuranishi space equipped with a hypercovering.
\end{thm}

\begin{proof}
Given a Kuranishi space with a hypercovering and a choice of representing atlas, we construct a map $\mathscr{G}^{(l)}_{\bullet} :  N\left(\widehat{\mathcal{U}}\right)_{\bullet} \rightarrow N_{\bullet}\left(\mathcal{K}_\mathfrak{X}^{(l)}\right)$ for each $l,$ inductively on the degree:
\begin{enumerate}
\item For $m=0,$ we assign to each element of $N\left(\widehat{\mathcal{U}}\right)_0$ (that is, to each point $p \in X$) a Kuranishi chart, as follows:
\begin{equation}\nonumber
\mathscr{G}^{(l)}_0(p) :=
\begin{cases}
{\mathcal{U}}_p  &\text{ if } {\mathcal{U}}_p \in \mathrm{Ob}\left(\mathcal{K}_\mathfrak{X}^{(l)}\right),\\
\overline{\mathcal{U}}_p  &\text{ if } {\mathcal{U}}_p \notin \mathrm{Ob}\left(\mathcal{K}_\mathfrak{X}^{(l)}\right),
\end{cases}
\end{equation}
where, in the second case, $\overline{\mathcal{U}}_p$ is taken to be a chart from an atlas $\widehat{\overline{\mathcal{U}}}$ such that (i) $\widehat{\mathcal{U}} \sim \widehat{\overline{\mathcal{U}}}$ and (ii) $\dim\overline{U}_p$ is smallest among all such $\widehat{\overline{\mathcal{U}}}.$
\item Suppose $\mathscr{G}^{(l)}_0(\alpha)$ has been defined for each $\alpha \in N\left(\widehat{\mathcal{U}}\right)_{0}.$  For $m=1$ and $\alpha \in N\left(\widehat{\mathcal{U}}\right)_1,$ we set
\begin{equation}\nonumber
\mathscr{G}^{(l)}_1(\alpha) := \left\{\big(\overline{U}_{\alpha}, {\Phi}^k_{v_0(\alpha)v_1(\alpha)}, \{\widehat{\Phi}^k_{v_0(\alpha)v_1(\alpha),x}\}\big)\right\}_{k \geq 0}
\end{equation}
with
\begin{equation}\label{htp1}
\begin{split}
\overline{U}_{\alpha} &:= \overline{U}_{v_0(\alpha)v_1(\alpha)},\\
\Phi_{v_0(\alpha)v_1(\alpha)}^{k} &: \overline{U}_{\alpha} \times \Delta^{k} \rightarrow \overline{U}_{v_1(\alpha)}, \text{ and }\\ \widehat{\Phi}_{v_0(\alpha)v_1(\alpha),x}^{k} &: \overline{\mathcal{C}}'_{v_1(\alpha), \Phi_{v_0(\alpha),v_1(\alpha)}^0(x)} \rightarrow \Omega^* (\Delta^{k}) \otimes \overline{\mathcal{C}}_{v_0(\alpha), x},\\  &\quad \quad \quad \quad \quad \quad \quad \text{ for each }x \in \overline{s}_{v_0(\alpha)}^{-1}(0) \cap \overline{U}_{\alpha}.
\end{split}
\end{equation}
(See Example \ref{hhex} for the $L_{\infty}[1]$-structure on the target of $\widehat{\Phi}_{v_0(\alpha)v_1(\alpha),x}^{k}$ in (\ref{htp1}).) Each map is defined as follows:
\begin{enumerate}[label = (\roman*)]
\item $\left(\overline{U}_{\alpha}, {\Phi}^0_{v_0(\alpha)v_1(\alpha)}, \left\{\widehat{\Phi}^0_{v_0(\alpha)v_1(\alpha),x}\right\}\right)$ is the coordinate change for the Kuranishi atlas $\overline{\mathcal{U}}.$
\item For $k \geq 1,$ we set
\begin{equation}\nonumber
\begin{split}
\Phi_{v_0(\alpha)v_1(\alpha)}^k (y, \vec{t}) &:= \Phi_{v_0(\alpha)v_1(\alpha)}^0(y),\\ &\text{ for each } y \in \overline{U}_{\alpha} \text{ and }\vec{t} \in \Delta^k,\\
\widehat{\Phi}_{v_0(\alpha)v_1(\alpha),x}^k(\xi) &:= \widehat{\Phi}_{v_0(\alpha)v_1(\alpha),x}^0(\xi),\\ &\text{ for each } x \in \overline{s}^{-1}_{v_0(\alpha)}(0) \cap \overline{U}_{\alpha} \text{ and }\xi \in \overline{\mathcal{C}}'_{v_1(\alpha), \Phi_{v_0(\alpha),v_1(\alpha)}^0(x)}.
\end{split}
\end{equation}
\end{enumerate}
\item Suppose $\mathscr{G}^{(l)}_m$ has been defined for $\alpha \in N\left(\widehat{\mathcal{U}}\right)_{m \leq 1}.$  For $m=2$ and $\alpha = (\alpha_0, \alpha_1, \alpha_2) \in N\left(\widehat{\mathcal{U}}\right)_2,$
 we set \begin{equation}\nonumber
\mathscr{G}^{(l)}_2(\alpha) := \left\{\left(\overline{U}_{\alpha}, \Phi_{\alpha}, \left\{\widehat{\Phi}_{\alpha, x}\right\}\right)\right\}_{k \geq 0}
\end{equation}
with
\begin{equation}\nonumber
\begin{split}
\Phi_{\alpha}^{k} &: \overline{U}_{\alpha} \times \Delta^{k} \rightarrow \overline{U}_{v_2(\alpha)}, \text{ and}\\ \widehat{\Phi}_{\alpha,x}^{k} &: \overline{\mathcal{C}}'_{v_2(\alpha), \Phi_{\alpha}^0(x)} \rightarrow \Omega^* (\Delta^{k}) \otimes \overline{\mathcal{C}}_{v_0(\alpha), x},\\ 
& \quad \quad \quad \quad \quad \quad \quad \quad \quad \quad \quad \quad \text{ for each } x \in \overline{s}_{v_0(\alpha)}^{-1}(0) \cap \overline{U}_{\alpha}.
\end{split}
\end{equation}
Each component is defined as follows:
\begin{enumerate}[label = (\roman*)]
\item We set
\[
\overline{U}_{\alpha} := \overline{U}_{v_0(\alpha)v_2(\alpha)} \cap (\Phi_{v_0(\alpha)v_1(\alpha)}^0)^{-1} (\overline{U}_{v_1(\alpha)v_2(\alpha)}),
\]
which is contractible by Assumption \ref{contrassmp}.
\item $\Phi^0_{\alpha}$ is a homotopy between
\begin{equation}\nonumber
\begin{split}
\Phi^0_{\alpha}|_{t=0} &= \Phi^0_{v_0(\alpha) v_1(\alpha)} \circ \Phi^0_{v_1(\alpha) v_2(\alpha)}|_{U_{\alpha}}, \text{ and}\\
\Phi^0_{\alpha}|_{t=1} &=  \Phi^0_{v_0(\alpha) v_2(\alpha)}|_{U_{\alpha}},
\end{split}
\end{equation}
whose existence is guaranteed by the contractibility of $\overline{U}_{\alpha}$.
\item
$\widehat{\Phi}^0_{\alpha,x}$ is an $L_{\infty}[1]$-homotopy between
\begin{equation}\nonumber
\begin{split}
\Eval_{0} \circ \widehat{\Phi}^0_{\alpha,x} &= \widehat{\Phi}^0_{v_0(\alpha)v_2(\alpha), x}, \text{ and}\\
\Eval_{1} \circ \widehat{\Phi}^0_{\alpha,x} &= \widehat{\Phi}^0_{{v_0(\alpha)v_1(\alpha)}, x} \circ \widehat{\Phi}^0_{{v_1(\alpha)v_2(\alpha)}, \Phi^0_{v_0(\alpha)v_1(\alpha)}(x)},
\end{split}
\end{equation}
at each $x \in \overline{s}_{v_0(\alpha)}^{-1}(0) \cap \overline{U}_{\alpha},$ whose existence is guaranteed by Theorem \ref{kcpx} and Corollary \ref{anhp}. Note that both 
\[
\widehat{\Phi}^0_{v_0(\alpha)v_2(\alpha),x}\text{ and }\widehat{\Phi}^0_{{v_1(\alpha)v_2(\alpha)}, \Phi^0_{v_0(\alpha)v_1(\alpha)}(x)} \circ \widehat{\Phi}^0_{{v_0(\alpha)v_1(\alpha)}, x}
\] 
are quasi-isomorphisms, by the definition of coordinate changes. (See Example \ref{hhex} for the definition of the maps $\Eval_{i}, \ i=0,1.$)
\item For $k \geq 1,$ we set
\begin{equation}\nonumber
\begin{split}
\Phi_{\alpha}^k (y, \vec{t}) &:= \Phi_{v_0(\alpha)v_1(\alpha)}^0 (y),\\ \quad \quad \quad \quad \quad \quad &\text{ for each } y \in \overline{U}_{\alpha} \text{ and } \vec{t} \in \Delta^k,\\
\widehat{\Phi}_{\alpha, x}^k(\xi) &:= \widetilde{\Phi}_{\alpha, x}^0(\xi),\\ \quad \quad \quad \quad \quad \quad &\text{ for each } \xi \in \mathcal{C}'_{v_2(\alpha), \Phi_{v_0(\alpha),v_2(\alpha)}^0(x)} \text{ and } x \in \overline{s}^{-1}_{v_0(\alpha)}(0) \cap \overline{U}_{\alpha}.
\end{split}
\end{equation}
\end{enumerate}
Conditions (i) and (ii) of Definition \ref{defhtpgl} are satisfied by construction.

\item
Suppose $\mathscr{G}^{(l)}_{m}$ has been defined for $m \leq 2.$ We now construct $\mathscr{G}^{(l)}_m$ for $m = 3.$

Let $\alpha \in N\left(\widehat{\mathcal{U}}\right)_2.$ Given the vertices $\mathscr{G}_0^{(l)}\big(v_j(\alpha)\big), \ j =0,1,2,3,$ we may fill in the edges $\mathscr{G}^{(l)}_1\big(v_j(\alpha)v_{j'}(\alpha)\big)$ as in (iii) above, obtaining the following diagram:
\begin{equation}\nonumber
\begin{tikzcd}
{} & \mathscr{G}^{(l)}_0\big(v_0(\alpha)\big) \arrow{rd}{\mathscr{G}^{(l)}_1\big(v_0(\alpha)v_1(\alpha)\big)} \arrow{ld}[swap]{\mathscr{G}^{(l)}_1\big(v_0(\alpha)v_2(\alpha)\big)} \arrow{dd}{\mathscr{G}_1^{(l)}\big(v_0(\alpha)v_3(\alpha)\big)} & {}\\
\mathscr{G}^{(l)}_0\big(v_1(\alpha)\big) \arrow{rd}[swap]{\mathscr{G}^{(l)}_1\big(v_1(\alpha)v_3(\alpha)\big)} & {} & \mathscr{G}^{(l)}_0\big(v_2(\alpha)\big) \arrow{ld}{\mathscr{G}^{(l)}_1\big(v_2(\alpha)v_3(\alpha)\big)}\\
{} & \mathscr{G}^{(l)}_0\big(v_3(\alpha)\big) & {}\\
\end{tikzcd}
\end{equation}
We denote by $\widehat{\Phi}_{v_j(\alpha)v_{j'}(\alpha), x}$

To fill this diagram with homotopies, we choose
\begin{equation}\nonumber
\mathscr{G}^{(l)}_3(\alpha) := \left\{\left(\overline{U}_{\alpha}, \Phi_{\alpha}, \left\{\widehat{\Phi}_{\alpha, x}\right\}\right)\right\}_{k \geq 0}
\end{equation}
with
\begin{equation}\nonumber
\begin{split}
\Phi_{\alpha}^{k} &: \overline{U}_{\alpha} \times \Delta^{k} \rightarrow \overline{U}_{v_3(\alpha)},\\ \widehat{\Phi}_{\alpha,x}^{k} &: \overline{\mathcal{C}}'_{v_3(\alpha), \Phi_{\alpha}^0(x)} \rightarrow \Omega^* (\Delta^{k}) \otimes \overline{\mathcal{C}}_{v_0(\alpha), x}, \text{ for each } x \in \overline{s}_{v_0(\alpha)}^{-1}(0) \cap \overline{U}_{\alpha}.
\end{split}
\end{equation}
Each component is defined as follows:
\begin{enumerate}[label = (\roman*)]
\item We set
\begin{equation}\nonumber
\begin{split}
\overline{U}_{\alpha} := \overline{U}_{v_0(\alpha)v_3(\alpha)} &\cap (\Phi_{v_0(\alpha)v_1(\alpha)}^0)^{-1} (\overline{U}_{v_1(\alpha)v_3(\alpha)})\\ &\cap (\Phi_{v_0(\alpha)v_2(\alpha)}^0)^{-1} (\overline{U}_{v_2(\alpha)v_3(\alpha)}),
\end{split}
\end{equation}
which is contractible by Assumption \ref{contrassmp}. Note that $\overline{U}_{\alpha}$ is always nonempty.
\item $\Phi^0_{\alpha}$ is a homotopy characterized by
\begin{equation}\nonumber
\Phi^0_{\alpha}|_{\overline{U}_{\alpha} \times d_i (\Delta^{2})} = \mathscr{G}^{(l)}(\partial_i \alpha), \  \ i=0,1,2,\\
\end{equation}
whose existence is guaranteed by the contractibility of $\overline{U}_{\alpha}$.
\item
$\widehat{\Phi}^0_{\alpha,x}$ is an $L_{\infty}[1]$-homotopy characterized by
\begin{equation}\nonumber
\Eval_{J} \circ \widehat{\Phi}^0_{\alpha, x} =
\begin{cases}
\widehat{\Phi}^0_{v_0(\alpha)v_1(\alpha), x} \circ \widehat{\Phi}^0_{v_1(\alpha)v_3(\alpha), {\Phi}_{v_0(\alpha)v_1(\alpha)}(x)} \ &\text{ if } J =\{0,1,3\},\\
\widehat{\Phi}^0_{v_0(\alpha)v_2(\alpha), x} \circ \widehat{\Phi}^0_{v_2(\alpha)v_3(\alpha), {\Phi}_{v_0(\alpha)v_2(\alpha)}(x)} \ &\text{ if } J =\{0,2,3\},\\
\widehat{\Phi}^0_{v_0(\alpha)v_3(\alpha), x} \ &\text{ if } J = \{0,1,2,3\},
\end{cases}
\end{equation}
at each $x \in \overline{s}_{v_0(\alpha)}^{-1}(0) \cap \overline{U}_{\alpha}.$ (See Example \ref{hhex} for the definition of the map $\Eval_{J}$ for $J \subset \{0,1,2,3\}.$) Its existence is guaranteed by Theorem \ref{kcpx} and Corollary \ref{anhp}.
\item For $k \geq 1,$ we set
\begin{equation}\nonumber
\begin{split}
\Phi_{\alpha}^k (y, \vec{t}) &:= \Phi_{v_0(\alpha)v_1(\alpha)}^0 (y),\\ & \quad  \quad \text{ for each } y \in \overline{U}_{\alpha}, \ \vec{t} \in \Delta^k, \text{ and}\\
\widehat{\Phi}_{\alpha, x}^k(\xi) &:= \widehat{\Phi}_{\alpha, x}^0(\xi),\\ &\quad \quad \text{ for each } \xi \in \overline{\mathcal{C}}'_{v_3(\alpha), \Phi_{v_0(\alpha),v_3(\alpha)}^0(x)} \text{ and } x \in \overline{s}_{v_0(\alpha)}^{-1}(0) \cap \overline{U}_{\alpha}.
\end{split}
\end{equation}
\end{enumerate}
Conditions (i) and (ii) of Definition \ref{defhtpgl} are satisfied by construction.
\end{enumerate}
The same construction clearly applies for all $m \geq 4,$ yielding the desired higher cocycle conditions. Conditions (i) and (ii) of Definition \ref{defhtpgl} are then satisfied by construction, and condition (iii) is immediate from the definition of the functor $\mathscr{I}^{(l)}.$
\end{proof}

\appendix

\section{$L_{\infty}[1]$-algebras and their higher homotopy theory}

In this section, we briefly introduce $L_{\infty}[1]$-algebras and their higher homotopy theory, following \cite{Kim2}. We first recall the notion of \textit{graded symmetric algebra} $S C$ of a vector space $C$ over a field $\mathbf{k},$
$$
S C := TC/ \langle v \otimes v' - (-1)^{|v|\cdot|v'|} v' \otimes v \rangle,
$$
with its degree $k$ component $S^k C := \{v \in S C \mid v \text{ is homogeneous of degree } k\}.$
We have a decomposition
$$
S C = \bigoplus\limits_{k=0}^{\infty} S^k C
$$
with the induced product $\odot$ on each component.
We denote by $\text{Sh}(i, k-i)$ the set of $(i, k-i)$-unshuffles, and the sign $\sgn(\tau)$ for $\tau \in \text{Sh}(i, k-i)$ is defined for homogeneous elements $a_1, \cdots, a_k \in C$, we write
$$
a_{\tau(1)} \odot \cdots \odot a_{\tau(k)} = \sgn(\tau) a_1 \odot \cdots \odot a_k.
$$

\begin{defn}
An \textit{$L_{\infty}[1]$-algebra} is a pair $\left(C, \{l_k\}\right)$ consisting of a vector space $C$ and a family of degree 1 linear maps
$$
l_k : S^k C \rightarrow C, \ k \geq 0,
$$
satisfying the relations
\begin{equation}\label{quadrel}
\sum\limits_{i = 0}^{k} \sum\limits_{\tau \in {\text{Sh}}(i, k-i)} {\sgn(\tau)} l_{k-i+1}\left(l_i(a_{\tau(1)}, \cdots, a_{\tau(i)}),a_{\tau(i+1)}, \cdots, a_{\tau(k)}\right) = 0.
\end{equation}
\end{defn}

\begin{defn}
Let $(C,\{l_k\})$ and $(C', \{l'_k\})$ be two $L_{\infty}[1]$-algebras. An $L_{\infty}[1]$\textit{-algebra morphism}, or simply $L_{\infty}[1]$\textit{-morphism}
\begin{equation}\label{lrel}
f : C \rightarrow C'
\end{equation}
is a family of degree 0 linear maps
$$
f_k : S^kC \rightarrow C', \ k \geq 0,
$$
satisfying the relations
\begin{equation}\label{frel}
\begin{split}
\sum\limits_{i = 0}^{k}& \sum\limits_{\tau \in {\text{Sh}}(i, k-i)} {\sgn(\tau)} f_{k-i+1}\left(l_i(a_{\tau(1)}, \cdots, a_{\tau(i)}),a_{\tau(i+1)}, \cdots, a_{\tau(k)}\right)\\
&= \sum\limits_{\substack{t, j_1, \cdots, j_t \geq 1,\\ j_1 + \cdots + j_t = k}} \sum\limits_{\tau \in S_k}  \frac{\sgn(\tau)}{t! j_1! \cdots j_t!} \  l'_{t}\bigl(f_{j_1}(a_{\tau(1)}, \cdots, a_{\tau(j_1)}), \cdots, \\
& \quad \quad \quad \quad \quad \quad \quad \quad \quad \quad \quad \quad f_{j_t}(a_{\tau(k -(j_1 + \cdots + j_{t-1}))}, \cdots, a_{\tau(k)})\bigr).
\end{split}
\end{equation}
Here, $S_k$ denotes the symmetric group of permutations of $k$ elements. 
\end{defn}

\begin{defn}
For two $L_{\infty}[1]$-morphisms
$$
f : C \rightarrow C', \ g: C' \rightarrow C'',
$$
we define their \textit{composition}
$$
g \circ f : C \rightarrow C''
$$
by a family of linear maps of degree 0 for $k \geq 0$
\begin{equation}\nonumber
\begin{split}
(g \circ f)_k := &\sum\limits_{i = 0}^{k} \sum\limits_{\tau \in S_k}  \frac{\sgn(\tau)}{t! j_1! \cdots j_t!} \ g_{t}\bigl(f_{j_1}(a_{\tau(1)}, \cdots, a_{\tau(j_1)}), \cdots,\\
&\quad \quad\quad \quad\quad \quad\quad \quad f_{j_t}(a_{\tau(k -(j_1 + \cdots + j_{t-1}))}, \cdots, a_{\tau(k)})\bigr).
\end{split}
\end{equation}
It is straightforward to verify that $\{(g \circ f)_k\}_{k \geq 0}$ satisfies the relation (\ref{frel}).
\end{defn}

\begin{defn}
We say an $L_{\infty}[1]$-algebra $\{l_k\}_{k \geq 0}$ is \textit{strict} if $l_0 = 0.$ Otherwise, we say it is \textit{curved.} We similarly define \textit{strict/curved} $L_{\infty}[1]$-morphisms.
\end{defn}

In the strict case, the relations (\ref{lrel}) and (\ref{frel}) coincide with the differential and the chain map relations, respectively. That is, they satisfy
$$
\begin{cases}
l_1 \left( l_1(a) \right) = 0,\\
l'_1 \left(f_1 (a) \right) = f_1 \left( l_1(a) \right).
\end{cases}
$$
\begin{defn}
We say that a strict $L_{\infty}[1]$-algebra $(C,\{l_k\})$ is \textit{acyclic} if its cohomology for each degree vanishes, that is, if 
\[
H^*(C) = \frac{\ker{l_1}}{\mathrm{Im}l_1} = 0.
\]
We say that a strict $L_{\infty}[1]$-morphism $\{f_k\}_{k \geq 1}$ between strict $L_{\infty}[1]$-algebras is a \textit{quasi-isomorphism} if $f_1$ is a quasi-isomorphic chain map.
\end{defn}
From now on and elsewhere in this paper, we always assume the strictness of $L_{\infty}[1]$-algebras.

As a preparation for the definition of higher homotopies of $L_{\infty}[1]$-algebras, we need to consider the following notion: 
\begin{defn}\cite[Definition 4.1]{Kim2}\label{hhtp}
Let $C$ be an $L_{\infty}[1]$-algebra. Suppose that we have defined models of $\Delta^k \times C$ with $k \leq n-1.$ We recursively define \textit{models of} $\Delta^n \times C$ with $n \geq 2$ to be a collection of $L_{\infty}[1]$-algebras
\begin{equation}\nonumber
\mathfrak{C}^{(n)}, \  \left(\mathfrak{C}^{(n)}\right)_{J} \equiv \mathfrak{C}_{J}^{(n-1)}, \text{ where } J \text{ is a subset of } \{0, \cdots, n\} \text{ such that } |J| = n,
\end{equation}
together with an $L_{\infty}[1]$-morphism
\[
\Eval_J^{(n)} :  \mathfrak{C}^{(n)} \rightarrow \mathfrak{C}_{J}^{(n-1)},
\]
and a chain map
\[
\Incl^{(n)} : {C} \rightarrow \mathfrak{C}^{(n)}\\
\]
with the following properties:
\begin{enumerate}[label = (\roman*)]
\item $\mathfrak{C}_{J}^{(n-1)}$ is a model of $\Delta^{n-1} \times C$ with $\mathfrak{C}_{\{i\}}^{(0)} = C$ for each $i.$
\item $\left(\mathfrak{C}_{J}^{(n-1)} \right)_{J'} = \left(\mathfrak{C}_{J'}^{(n-1)} \right)_{J} = \mathfrak{C}_{J \cap J'}^{(n-2)}$ for all $J, J' \subset \{0, \cdots, n\}$ with $|J| = |J'| = n$ and $|J \cap J'| = n-1.$ 
\item $\left(\Eval_j\right)_{k \geq 2} \equiv 0, \ j =0,1, \ \Incl_{k \geq 2} \equiv 0.$
\item $\Eval_J^{(n)}$ and $\Incl^{(n)}$ are quasi-isomorphisms.
\item $\Eval_J^{(n)} \circ \Incl^{(n)} = \Incl_J^{(n-1)},$
where $ \Incl_J^{(n-1)}$ is the Incl map for $\mathfrak{C}_{J}^{(n-1)},$ the model of $\Delta^{n-1} \times C$ for the index $J.$
\item The following sequence of chain complexes
\[
\begin{split}
\mathfrak{C}^{(n)} \xrightarrow{\partial_n}  \bigoplus_{\substack{J \subset \{0, \cdots, n\}, \\ |J| = n}} &\mathfrak{C}_J^{(n-1)} \xrightarrow{\partial_{n-1}} \bigoplus_{\substack{J' \subset \{0, \cdots, n\}, \\ |J'|=n-1}} \mathfrak{C}_{J'}^{(n-2)} \xrightarrow{\partial_{n-2}} \\
&\cdots \xrightarrow{\partial_{2}} \bigoplus_{\substack{J'' \subset \{0, \cdots, n\}, \\ |J''|=2}} \mathfrak{C}_{J''}^{(1)} \xrightarrow{\partial_{1}} \bigoplus_{i \in \{0, \cdots, n\}}C \rightarrow 0
\end{split}
\]
is in fact a chain complex that is exact at the first term.
Here, the differentials $\partial_n$ and $\partial_{n-k}, \ 1 \leq k \leq n-1$ are given by
\end{enumerate}
\end{defn}

Using models of $\Delta^n \times C$, we can define higher homotopies:
\begin{defn}\cite[Definition 4.2]{Kim2}
Let $C$ and $C'$ be $L_{\infty}[1]$-algebras and $f_0, \cdots, f_n : C \rightarrow C',$ $L_{\infty}[1]$-morphisms for $n \geq 1.$ Consider a sequence $\vec{J}$ of subsets
\begin{equation}\nonumber
\vec{J} : J_0 \subsetneq J_1 \subsetneq \cdots \subsetneq J_{n-1} \subsetneq \{0, \cdots, n\}
\end{equation}
with $|J_l| = l+1, \ 0 \leq l \leq n-1.$
We say $f_0, \cdots, f_n : C \rightarrow C'$ are $n$-\textit{homotopic} if there exist a model of $\Delta^n \times C',$ say $\mathfrak{C}'^{(n)},$ and an $L_{\infty}[1]$-morphism $h : C \rightarrow \mathfrak{C}'^{(n)}$ such that
\begin{equation}\nonumber
\Eval^{(1)}_{J_0} \circ \cdots \circ \Eval^{(n)}_{J_{n-1}} \circ h = f_j
\end{equation}
for each sequence $\vec{J}$ with $J_0 = \{j\}.$
We call such a map $h$ an \textit{n-homotopy} ($n \geq 1$) of $f_0, \cdots, f_n.$ 
\end{defn}

\begin{exam}\label{hhex}
Let $\Delta^n$ be the standard $n$-simplex and $\Omega^*(\Delta^n)$ its de Rham complex over a field. We denote
\begin{equation}\nonumber
\begin{split}
\mathfrak{C}^{(n)} &:= \Omega^*(\Delta^n) \otimes C,\\
\mathfrak{C}^{(n-1)}_{J_i} &:= \Omega^*(\partial_i \Delta^n) \otimes C, \text{ where } \ J_i = \{0, \cdots, n\} \setminus \{i\} \text{ for } 0 \leq i \leq n.\
\end{split}
\end{equation}
On $\mathfrak{C}^{(n)},$ there exists an $L_{\infty}[1]$-algebra structure
\[
l_k : (\mathfrak{C}^{(n)})^{\otimes k} \rightarrow \mathfrak{C}^{(n)}, \ k \geq 1,
\]
which is given by
\begin{equation}\nonumber
l_k(\alpha_1 \otimes x_1, \cdots, \alpha_k \otimes x_k) := 
\begin{cases}
d \alpha_1 \otimes x_1 + (-1)^{|\alpha_1|} \alpha_1 \otimes l_1(x_1) &\text{ if } k =1,\\
 (-1)^{|\vec{\alpha}|} \alpha_1 \wedge \cdots \wedge \alpha_k \otimes l_k(x_1, \cdots, x_k) &\text{ if } k \geq 2,
\end{cases}
\end{equation}
for each $\alpha_i \in \Omega^*(\Delta^n), \ x_i \in C, \ i = 1, \cdots, k,$ and $k \geq 1.$ Here, we denote
\begin{equation}\nonumber
|\vec{\alpha}| := \sum\limits_{i = 1}^{k-1}|x_i| \cdot (|\alpha_{i+1}|+ \cdots + |\alpha_k|) + \sum\limits_{i = 1}^k |\alpha_i|.
\end{equation}

The $L_{\infty}[1]$-morphism $\Eval^{(n)}_{J_i} = \left\{\left(\Eval^{(n)}_{J_i}\right)_k\right\}_{k \geq 1}$ and  and the chain map $\Incl^{(n)}$ are given by
\begin{equation}\nonumber
\left(\Eval_{J_i}^{(n-1)}\right)_k : \left(\Omega^*(\Delta^n) \otimes C\right)^{\otimes k} \rightarrow \Omega^*(\partial_i \Delta^n) \otimes C,
\end{equation}
\begin{equation}\nonumber
\left(\Eval_{J_i}^{(n-1)}\right)_k :=  
\begin{cases}
\text{(the restriction to $i$-th face)} \otimes \textrm{id}_C & \text{ if } k = 1,\\
 0 &\text{ if } k \geq 2,
\end{cases}
\end{equation}
and
\begin{equation}\nonumber
\begin{split}
\Incl^{(n)} &: C \rightarrow \Omega^*(\Delta^n) \otimes C,\\
\Incl^{(n)} &:= 1 \otimes \textrm{id}_C,
\end{split}
\end{equation}
respectively. It immediately follows that all the conditions in Definition \ref{hhtp} are satisfied. In particular, one can easily see that $\Eval^{(n)}_{J_i}$ and $\Incl^{(n)}$ are $L_{\infty}[1]$-algebra morphisms. Moreover, they are quasi-isomorphisms, whose proof can be sketched as follows. Since $\Incl^{(n)}$ is injective, its quasi-isomorphicity is equivalent to the acyclicity of the quotient complex 
\[
\frac{\Omega^*(\Delta^n) \otimes C}{\Incl^{(n)}(C)} \simeq \frac{\Omega^*(\Delta^n) \otimes C}{\{1\} \otimes C} \simeq \frac{\Omega^*(\Delta^n)}{\{\text{const. ftns.}\}} \otimes C,
\]
which follows from the acyclicity of $\frac{\Omega^*(\Delta^n)}{\{\text{const. ftns.}\}}$ and the K\"{u}nneth formula. Finally, the axiom (iv) of Definition \ref{hhtp} with an inductive argument implies that $\Eval^{(n)}_{J_i}$ is also a quasi-isomorphism.
\end{exam}

We now state a key proposition in this section:
\begin{prop}\cite[Proposition 4.5]{Kim2}\label{pphhe}
Let $f_0, \cdots, f_{n+1} : C_0 \rightarrow C \ (n \geq 0)$ be quasi-isomorphic $L_{\infty}[1]$-morphisms. Suppose that we are given an $n$-homotopy $h_J : C_0 \rightarrow \mathfrak{C}^{(n)}_J$ of $f_{j_0}, \cdots, f_{j_n}$
 for each given $J = \{ j_0 < \cdots < j_n\} \subset \{0, \cdots, n+1\},$ satisfying $\Eval^{(n)}_{J \cap J'} \circ h_J = \Eval^{(n)}_{J \cap J'} \circ h_{J'}$ for two distinct $J$ and $J'.$ Then there exist  a model $\mathfrak{C}^{(n+1)}$ of $\Delta^{n+1} \times C$ and an $(n+1)$-homotopy $\overline{h} : C_0 \rightarrow \mathfrak{C}^{(n+1)}$ of $f_0, \cdots, f_{n+1}$ such that ${\mathfrak{C}^{(n)}_J}$'s belong to the data for $\mathfrak{C}^{(n+1)},$ satisfying $\mathrm{Eval}^{(n+1)}_J \circ \overline{h} = h_J.$
\end{prop} 

We show that a key theorem on quasi-isoms, which shows the usefulness of our definition
\begin{cor}\cite[Corollary 4.6]{Kim2}\label{anhp}
Arbitrarily given quasi-isomorphic $L_{\infty}[1]$-morphisms $f_0, \cdots, f_n : C \rightarrow C' \ (n \geq 1)$ are $n$-homotopic.
\end{cor}

\section{Presymplectic Neighborhoods and Local $L_{\infty}[1]$-Algebras}

In this section, we provide the details on the structures that we put on a Kuranishi chart. For more detail, we refer the reader to \cite{Kim1}.

Let a Kuranishi chart be given with the base $U.$ We equip $U$ with a closed two-form $\beta \in \Omega^2(U).$ 

We require that $U$ has a decomposition 
\begin{equation}\label{aabbbb}
U = \bigcup\limits_i \mathcal{S}_i,
\end{equation}
into (possibly non-connected) submanifolds,
\[
\mathcal{S}_i := \{x \in U \mid \text{rk} (\ker\beta_x) =i \}, \ 0 \leq i \leq \dim U
\]
together with their tubular neighborhoods:
\[
\begin{cases}
\iota_{i} : N_i \rightarrow  U, \text{ an open neighborhood of each }  \mathcal{S}_{i} \text{ in }U,\\
\pi_{i} : N_i \rightarrow \mathcal{S}_{i}, \text{ the associated projection}.
\end{cases}
\]

\begin{rem}
\cite[Corollary 6.2]{KO} proves that there exists a residual subset of such closed two-forms. Furthermore, the stratification structure turns out to be Whitney (cf. \cite[Definition 6.5 \& Theorem 6.6]{KO}), which leads to the construction of \textit{stratified $L_{\infty}$-spaces.}
\end{rem}

\subsubsection*{The presymplectic open neighborhood $W_x$} We associate an open contractible submanifold $W_x$ to each zero point $x \in s^{-1}(0).$ 

For each zero point $x \in s^{-1}(0) \cap \mathcal{S}_i$ for some $i$, we choose $\overset{\circ}{W}_x \subset \mathcal{S}_i,$ an open ball containing $x$ in $\mathcal{S}_i.$ For the inclusion ${\iota_x} : \overset{\circ}{W}_x {\hookrightarrow} U,$ we denote $\beta|_{\overset{\circ}{W}_x} := \iota_x^*\beta.$ We have $d\left(\beta|_{\overset{\circ}{W}_x}\right) = d\iota_x^*\beta = \iota_x^* d\beta = 0$ and that $\beta|_{\overset{\circ}{W}_x}$ is of constant rank by construction. In other words, $\left(\overset{\circ}{W_x}, \beta|_{\overset{\circ}{W_x}}\right)$ is a presymplectic manifold.

Then we obtain another presymplectic manifold
\begin{equation}\label{adslfjkdh}
W_x = (W_x, \beta_{W_x}) := 
\left(\pi_i^{-1}(\overset{\circ}{W_x}), \pi^*_i(\beta|_{\overset{\circ}{W_x}})\right)
\end{equation}
and call it a \textit{local presymplectic neighborhood of} $x \in s^{-1}(0).$ 

We write 
\[
T\mathcal{F}_x := \ker \beta_{W_x}
\]
for the regular foliation (i.e., each leaf having the same dimension) determined by the kernel of $\beta_{W_x}.$ In \cite[Lemma B.2]{Kim1}, we equip $\Omega^{\bullet +1}(\mathcal{F}_x)$ with an $L_{\infty}[1]$-algebra structure $\{l^{\mathcal{F}}_k\}.$ Furthermore, we have:

\subsubsection*{The $L_{\infty}[1]$-algebra $\mathcal{C}_x$}
At each zero point $x \in s^{-1}(0),$ we associate a \textit{local $L_{\infty}[1]$-algebra,}
\begin{equation}\nonumber 
\mathcal{C}_x := \overbrace{\bigwedge\nolimits^{-\bullet}\Gamma(E^*|_{W_x})}^{\text{Koszul}} \oplus \overbrace{\Omega^{\bullet + 1}_{\mathrm{aug}}(\mathcal{F}_x)}^{\text{de Rham}},
\end{equation}
which consists of the two parts: Koszul and de Rham.

The Koszul part, $\bigwedge\nolimits^{-\bullet}\Gamma(E^*|_{W_x})$ is the Koszul complex,
\[
0 \rightarrow \overbrace{\bigwedge\nolimits^{r}\Gamma(E^*|_{W_x})}^{\text{deg} = -r} \xrightarrow{\iota_{s|_{W_x}}} \cdots \xrightarrow{\iota_{s|_{W_x}}}  \overbrace{\Gamma(E^*|_{W_x})}^{\text{deg} = -1} \xrightarrow{\iota_{s|_{W_x}}}  \overbrace{C^{\infty}(W_x)}^{\text{deg} = 0} \rightarrow 0
\]
with the differential $l_1^{\mathrm{K}} := \iota_{s|_{W_x}},$ given by:
\[
l_1^{\mathrm{K}} : a_1 \wedge \cdots \wedge a_m \mapsto \sum\limits_{i=1}^{m} (-1)^{i+1}a_i(s|_{W_x}) \cdot a_1 \wedge \cdots \wedge \widehat{a_i} \wedge \cdots \wedge a_m,
\]
with all higher $l^{\mathrm{K}}_{k \geq 2}$ being set to zero.

The de Rham part, $\Omega^{\bullet + 1}_{\mathrm{aug}}(\mathcal{F}_x)$ is the augmented foliation de Rham complex degree shifted by 1 equipped with the $L_{\infty}[1]$-algebra structure $\{l^{\mathrm{dR}}_k\}_{k \geq 1}$ obtained from Proposition \ref{augomega}. 

\begin{prop}\cite[Corollary 2.10, Proposition B.16]{Kim1}\label{augomega}
In the above situation, there exists an $L_{\infty}[1]$-algebra structure on the chain complex $\Omega^{\bullet +1}_{\mathrm{aug}}(\mathcal{F}_x)$ that extends $\{l^{\mathcal{F}}_k\}$ on $\Omega^{\bullet +1}(\mathcal{F}_x).$ Moreover, this $L_{\infty}[1]$-algebra has trivial cohomology.
\end{prop}

The $L_{\infty}[1]$-structure on $\mathcal{C}_x$ is then given by
\[
l_k : \mathcal{C}_x^{\otimes k} \rightarrow \mathcal{C}_x; \ l_k := l^{\mathrm{K}}_k \oplus l^{\mathrm{dR}}_k,
\]
where the direct sum notation indicates that the operations on the two components are defined separately. It is immediate that the family $\{l_k\}_{k \geq 1}$ satisfies the $L_{\infty}[1]$-relation.

\begin{lem}\cite[Lemma 2.3]{Kim1}\label{lemdw}
For different choices of $\overset{\circ}{W}_x$, we obtain isomorphic de Rham $L_\infty[1]$-algebras.  
\end{lem}

For the definition of chart morphisms we need the notion of completed algebras:
\begin{defn}[Completed algebras]\label{ladef}
Given a smooth map $\phi: V \rightarrow U,$ We define the \textit{completion of} $\mathcal{C}_{x}$ by
\begin{equation}\nonumber 
\mathcal{C}_{x, \phi} := \left(\bigwedge\nolimits^{-\bullet}\Gamma(E^*|_{W_x})\right)_{\phi} \oplus \Omega^{\bullet + 1}_{\mathrm{aug}, \phi}(\mathcal{F}_x),
\end{equation}
where the Koszul part
\begin{equation}\label{kdzpt}
\left(\bigwedge\nolimits^{-\bullet}\Gamma(E^*|_{W_x})\right)_{\phi} :=  C^{\infty}_{\phi}(W_x) \otimes_{C^{\infty}(W_x)} \bigwedge\nolimits^{-\bullet}\Gamma(E^*|_{W_x}),
\end{equation}
where we consider the inverse limit
\begin{equation}\label{ivlmt}
C_{\phi}^{\infty}(W) := \lim_{\longleftarrow} C^{\infty}(W_x)/ I_{\phi}^{j} \cdot C^{\infty}(W_x),
\end{equation}
where we denote $I_{\phi} := \left\{ f \in C^{\infty}(W_x) \mid f|_{\text{Im}\phi} \equiv 0\right\}.$

Its $L_{\infty}[1]$-structure
\[
l^{\mathrm{K}, \phi}_k: {\left(\bigwedge\nolimits^{i}\Gamma(E^*|_{W_x})\right)_{\phi}}^{\otimes k} \rightarrow \left(\bigwedge\nolimits^{i-1}\Gamma(E^*|_{W_x})\right)_{\phi},\\
\]
for each $1 \leq i \leq r$ is defined as follows: For each $j\geq 1,$ $h \in C^{\infty}(W_x)^{(j)},$ and $a \in \bigwedge\nolimits^{-\bullet}\Gamma(E^*|_{W_x}),$ we set
\[
(j) : l^{\mathrm{K}, \phi}_1(h \otimes a) := [1]_{j-2} \otimes \iota_{s|_{W_x}}(\widetilde{h}a),
\]
where $\widetilde{h}$ is a choice of representative in $C^{\infty}(W_x)$ such that $h = \widetilde{h} + I_{\phi}^j,$ and we set $[1]_{j-2} := 0$ for $j \leq 2$ by definition. All higher $l_{\phi, k \geq 2}^{\mathrm{K}}$'s are set to zero, so the $L_{\infty}$-relation of $\{l^{\mathrm{K}}_{\phi, k}\}_{k \geq 1}$ holds trivially.

The de Rham part $\Omega^{\bullet + 1}_{\mathrm{aug}, \phi}(\mathcal{F}_x)$ is the completed foliation de Rham complex \textit{with augmentation,} given by
\[
\Omega^{\bullet + 1}_{\mathrm{aug}, \phi}(\mathcal{F}_x) := \overbrace{\Omega^{\bullet + 1}(\mathcal{F}_x)_{\phi}}^{\deg \geq -1} \oplus \overbrace{\left(C^{\infty}(W_x)_{\mathcal{F}_x}\right)_{\phi}}^{\deg = -2},
\]
where we denote
\[
\begin{cases}
\Omega^{\bullet + 1}(\mathcal{F}_x)_{\phi} := C_{\phi}^{\infty}(W_x) \otimes_{C^{\infty}(W_x)} \Omega^{\bullet + 1}(\mathcal{F}_x),\\
\big(C^{\infty}(W_x)_{\mathcal{F}_x}\big)_{\phi} : = \ker \big(l_1^{\mathrm{dR}} : \Omega^{-1}(\mathcal{F}_x)[1]_{\phi} \rightarrow \Omega^{0}(\mathcal{F}_x)[1]_{\phi}\big).
\end{cases}
\]
The de Rham part $L_{\infty}[1]$-structure $l_{\phi, k}^{\mathrm{dR}}$ is obtained by applying the following Proposition \ref{augomega}. The formula can be given similarly to $l^{\mathrm{K}, \phi}_{\phi, k}$ in the setting of $V$-algebras in Appendix B of \cite{Kim1}. We omit its precise formulation here. For more details, we refer the reader to \cite[Definition B.19]{Kim1}.

In conclusion, $\mathcal{C}_{x,\phi}$ with $\left\{l_{\phi, k} := l^{\mathrm{K}}_{\phi, k} \oplus l^{\mathrm{dR}}_{\phi, k}\right\}$ is an $L_{\infty}[1]$-algebra with the $L_{\infty}[1]$-relation, which can be verified in a straightforward manner.
\end{defn}

Given a local algebra $\mathcal{C}_x,$ there exists a natural map to its completion:
For each $k \geq 1$, We define
\begin{equation}\label{varep}
\widehat{\varepsilon}_{\phi(x), \phi,k} : \mathcal{C}_x^{\otimes k} \rightarrow \mathcal{C}_{x, \phi}
\end{equation}
by
\begin{equation}\nonumber
\widehat{\varepsilon}_{\phi(x), \phi, k}\big((a_1, \xi_1), \cdots, (a_k,\xi_k)\big) := \begin{cases}
1 \otimes (a_1, \xi_1) = (1 \otimes a_1, 1 \otimes \xi_1) &\text{ if } k = 1,\\
0 &\text{ if } k \geq 2,
\end{cases}
\end{equation}
and consider the family $\widehat{\varepsilon}_{\phi(x), \phi} := \left\{\widehat{\varepsilon}_{\phi(x), \phi,k}\right\}_{k \geq 1}.$
\begin{lem}\cite[Lemma 2.8]{Kim1}\label{vecpf}
$\widehat{\varepsilon}_{\phi(x), \phi}$ is an $L_{\infty}[1]$-morphism.
\end{lem}

In the context of Definition \ref{ladef} regarding local $L_{\infty}[1]$-algebras, several special cases are worth discussing. First, we have a simple type of completions for surjective $\phi.$
\begin{lem}\cite[Lemma 2.12]{Kim1}\label{surjcp}
If $\phi$ is surjective, then we have an isomorphism
\[
\mathcal{C}_{\phi(x),\phi} \simeq \mathcal{C}_{\phi(x)}.
\]
\end{lem}

Second, we consider the completion for open subcharts: Let  $o : U \hookrightarrow U'.$ be an open inclusion and $\mathcal{U} := \mathcal{U}'|_U$ the open subchart on $U.$

\begin{lem}\cite[Lemma 2.13]{Kim1}\label{itavsteai}
In the above situation, there exists an $L_\infty[1]$-quasi-isomorphism:
\[
\widehat{o}_x : \mathcal{C}'_{o(x),o} \simeq \mathcal{C}_x.
\]
\end{lem}
\label{appen:forms}

\end{document}